\documentclass[12pt]{amsart}
\usepackage[margin=1.2in]{geometry}
\usepackage{color}
\usepackage[
  backend=biber,
  style=alphabetic,
  sorting=nyt
]{biblatex}
\title[Algebra of the Infrared with Elliptic Curve-Valued Potential]{Algebra of the Infrared with Elliptic Curve-Valued Potential}

\author[Longfei Li]{Longfei Li}

\usepackage{hyperref}

\date{\vspace{-5ex}}
\usepackage{amsthm, amsfonts, mathrsfs}
\usepackage{ dsfont }
\usepackage{amssymb}
\usepackage{enumerate}
\usepackage{graphicx}
\usepackage{subfig}
\usepackage{amscd}
\usepackage{enumitem}

\usepackage{float}
\usepackage{bbm}
\usepackage{amsmath}
\usepackage{comment}
\usepackage{hyperref}
\usepackage{listings}
\usepackage{color}
\usepackage[percent]{overpic}

\usepackage{diagbox}
\usepackage[dvipsnames]{xcolor}
\usepackage{algorithm}
\usepackage{booktabs}
\usepackage{algorithmicx}
\usepackage[noend]{algpseudocode}
\usepackage{tikz}
\allowdisplaybreaks

\hypersetup{
    colorlinks,
    citecolor=blue,
    filecolor=blue,
    linkcolor=blue,
    urlcolor=blue
}

\numberwithin{equation}{section}
\newtheorem{theorem}[equation]{Theorem}
\newtheorem{proposition}[equation]{Proposition}
\newtheorem{lemma}[equation]{Lemma}

\newtheorem{definition}[equation]{Definition}
\newtheorem{remark}[equation]{Remark}

\newtheorem{example}{Example}

\newtheorem{conjecture}[equation]{Conjecture}

\begin{document}
	\date{}

\address{\textrm{(Longfei Li)}
	Department of Mathematics,
	Kansas State University,
	Manhattan, Kansas 66502 USA}
\email{longfeili@ksu.edu}

\maketitle
\begin{abstract}
We study an extension of the algebra of the infrared to curve--valued potentials, focusing on the elliptic curve case. Given a finite configuration of points on an elliptic curve, we construct associated \(L_\infty\)-algebra. Adding an extra point produces a directed \(A_\infty\)-algebra and an \(L_\infty\)-morphism to their directed Hochschild deformation complexes. Maurer--Cartan elements of the \(L_\infty\)-algebra determine \(A_\infty\)-deformations of the directed algebra. We formulate a conjectural
comparison between Maurer--Cartan deformations of the directed
algebra and the directed Fukaya--Seidel category associated to
compatible finite collections of thimbles.

\end{abstract}

\tableofcontents
\section{Introduction}
\label{sec:introduction}

\subsection{Background}
The algebra of the infrared was introduced by Gaiotto--Moore--Witten
\cite{GMW15} in the study of massive two-dimensional
\(\mathcal N=(2,2)\) quantum field theories with finitely many vacua.
In this formalism, the vacua are labelled by a generic finite configuration
\(A\subset\mathbb C\), and complexes of soliton states provide
coefficients attached to pairs of labels. Planar webs encode operations
on these complexes: their faces carry vacuum labels, while their edges
have directions prescribed by the corresponding differences in
\(\mathbb C\).

The web identities organize an \(L_\infty\)-algebra. Additional
ordering data give a directed \(A_\infty\)-algebra or category on
which this \(L_\infty\)-algebra acts by higher derivations.
In the proposed geometric realization, instanton counts define a
Maurer--Cartan element and hence a deformation of the directed algebra.
For Landau--Ginzburg models, the expected comparison relates this
deformation to the Fukaya--Seidel category of the potential, the
symplectic category associated to its vanishing cycles and thimbles;
see \cite{GMW15,Seidel}.

Kapranov--Kontsevich--Soibelman interpreted this construction through
secondary polytopes \cite{KKS}. A finite affine point configuration
determines a secondary polytope whose faces encode regular marked
subdivisions. Factorization of these faces expresses cellular boundaries
as products of smaller secondary polytopes. The identity
\(\partial^2=0\) then produces the higher algebraic identities.
In dimension two, the construction gives an \(L_\infty\)-algebra
of polygons and, after choosing a half--plane, a relative directed
\(A_\infty\)-algebra. Mixed subdivisions define an
\(L_\infty\)-morphism to the latter's deformation complex.
The universality statement of \cite{KKS} concerns this morphism and
its relation to directed deformations.

Kapranov--Soibelman--Soukhanov also developed a categorical version
of the algebra of the infrared using perverse schobers \cite{KSS},
following the proposal of Kapranov--Schechtman
\cite{KapranovSchechtmanSchobers}. Categories and transport functors
replace vector spaces and linear maps, making the relation with
vanishing cycles and monodromy particularly natural. For a survey
of these developments, see \cite{KapranovSoibelmanSurveyAoI}.

\subsection{Algebra of the infrared with curve-valued potential}
The geometric realization considered in \cite{KKS}
starts from a complex-valued holomorphic Lefschetz fibration
\[
W:X\longrightarrow\mathbb C
\]
in a suitable symplectic setting. Its critical values provide
the planar configuration, and the proposed reconstruction of
its Fukaya--Seidel algebra is formulated in
\cite[Conjecture~14.10]{KKS}.

Following this direction, we consider a holomorphic Lefschetz
fibration
\[
f:X\longrightarrow Y,
\]
where \(Y\) is a smooth complex curve. We assume finitely many
nondegenerate critical points with distinct critical values,
together with the symplectic and analytic hypotheses specified
in Section~\ref{sec:mc-fs-torus}. In a local coordinate on \(Y\),
the map has the usual holomorphic Morse model. The global
geometry also involves the topology of the base: paths can
wind around its cycles, and symplectic parallel transport
along them can change the vanishing cycles and thimbles.
Choose a regular basepoint \(z\in Y\). An admissible path from
\(z\) to a critical value determines a vanishing cycle in
\(f^{-1}(z)\) and a thimble in \(X\). With compatible ordering
and Floer data, a finite collection of such paths gives a
directed \(A_\infty\)-category. A framework over a general
base curve is described in \cite[Section~5.2]{KS}.
Our question is how the algebra of the infrared can encode
this geometry when the potential takes values in a curve.

The elliptic curve \(E\) is particularly suitable for this question
because it combines a nonsingular translation structure with
a nontrivial fundamental group. A nonzero holomorphic
one-form \(\eta\in H^0(E,\Omega_E^1)\) supplies affine coordinates in which the polygonal
constructions of the algebra of the infrared can be carried
out. At the same time, winding around the cycles of \(E\)
can change period data and modify the transported vanishing cycles through
monodromy. The elliptic case thus provides a setting in which
the planar combinatorial framework can be extended to examine
how the topology of the base enters the associated algebras
and their proposed Fukaya--Seidel realization.

\subsection{Main algebraic results}
The algebraic construction starts with a marked affine polygon
\[
\mathcal P=(P,M,\rho),\qquad \rho:P\longrightarrow E,
\]
where \(P\) is a disk with a translation structure, its development
is a two-dimensional convex polygon, and
\(\rho|_M:M\to A\) is a bijection. The finite marking \(M\)
contains the vertices of \(P\). 
Secondary polytopes of marked subpolygons give a full
\(L_\infty\)-algebra and its geometric subalgebra
\(\mathfrak g_{\mathcal P}\). Coefficient systems replace the
orientation-line summands by complexes built from edge data, giving
\(\mathfrak g_{\mathcal N}\).

For the relative construction, choose \(p\in E\setminus A\) and
an augmented marked polygon
\[
\mathfrak d=(\mathcal P^+,r),\qquad
\mathcal P^+=(P^+,M\sqcup\{r\},\rho),\qquad \rho(r)=p.
\]
The root \(r\) is required to be a vertex of \(P^+\), and the
augmented configuration satisfies the same genericity hypotheses.
The rays from \(r\) to the other marks lie in a cone of angle less
than \(\pi\), so their angular order gives the required linear order.

Rooted polygons produce the directed algebra
\(R_{\mathcal N,\mathfrak d}\). On geometric summands this is a
dg associative algebra before deformation; its unary operation is
the coefficient differential. Adjoining the fixed diagonal algebras
gives a unital triangular algebra
\[
R^+=R^+_{\mathcal N,\mathfrak d}
=\bigoplus_{i\leq j}R_{ij},\qquad R_{ii}=S_i.
\]
Mixed subdivisions define an \(L_\infty\)-morphism to its
directed positive-arity Hochschild complex. 
Write
\[
\mathfrak h_{>0}(\mathfrak d,\mathcal N)
=\overrightarrow C^{\,\ge1}_{>0}(R^+,R^+)[1].
\]

\begin{theorem}[Elliptic universality; Theorem~\ref{universality-elliptic}]
\label{thm:intro-universality}
Let \(\mathfrak d\) be a finite relative datum on
\((E,\eta)\) satisfying the standing genericity and
factorization hypotheses. Assume \(S_i=\Bbbk\) and that the
edge complexes are bounded and finite-dimensional, with
perfect degree-zero chain pairings as in
Section~\ref{coeffecient-system}.
Mixed subdivisions define an \(L_\infty\)-morphism
\[
\Psi_{\mathfrak d}:\mathfrak g_{\mathcal N}
\longrightarrow
\mathfrak h_{>0}(\mathfrak d,\mathcal N).
\]
The morphism \(\Psi_{\mathfrak d}\) is a quasi-isomorphism.

Every Maurer--Cartan element \(\gamma\in\mathfrak g_{\mathcal N}^1\)
determines operations
\[
m_k^\gamma=m_k+
\sum_{n\ge1}\frac1{n!}
\bigl(\Psi_{\mathfrak d,n}(\gamma,\ldots,\gamma)\bigr)_k
\]
making the same graded spaces into a directed, uncurved,
strictly unital \(A_\infty\)-algebra with fixed diagonal units.
The sums are finite for the fixed finite datum.
\end{theorem}

For each fixed finite datum, affine development identifies
the algebraic operations with their planar counterparts.
The new dependence appears when the relative homotopy
datum changes. A period shift can change the convex geometry
and regular subdivisions even if the projected marked
points and angular order remain fixed.
Proposition~\ref{prop:winding-non-quasi-isomorphic} exhibits
two such data on the square elliptic curve with trivial
coefficients. Their undeformed rooted algebras satisfy
\[
\sum_k\dim_\Bbbk H^k(R^{(0)})T^k=3+2T+T^2,
\qquad
\sum_k\dim_\Bbbk H^k(R^{(1)})T^k=3+3T+T^2.
\]
Thus they are not \(A_\infty\)-quasi-isomorphic.
This gives an algebraic manifestation of winding dependence;
the example does not itself identify either algebra with
the Floer algebra of a fibration.

\subsection{Proposed Fukaya--Seidel realization}
For a holomorphic Lefschetz fibration \(f:X\to E\), the
one-form \(\eta\) supplies a global closed holomorphic one-form
\[
\alpha=f^*\eta.
\]
If \(w_U\) is a local primitive of \(\eta\), then
\(W_U=w_U\circ f\) is a local potential with \(dW_U=\alpha\).
These potentials differ by constants on connected overlaps,
so their gradient equations agree. Section~\ref{sec:mc-fs-torus}
uses \(\alpha\) to formulate soliton equations and candidate
instanton moduli spaces intrinsically on \(X\).

This description also gives a connection with the extension
of the algebra of the infrared to theories with twisted
masses studied by Khan--Moore \cite{KhanMoore2024}.
Their formulation in terms of closed holomorphic one-forms
allows nontrivial periods, with a detailed treatment of the case
of one vacuum and one twisted mass. It provides physical
motivation for retaining period and soliton-sector data
in the curve-valued setting.

Fix a regular basepoint \(z\), compatible brane and Floer
data, and a finite ordered collection of admissible paths
\(S=(\delta_1,\ldots,\delta_N)\) from \(z\) to the critical
values \(w_i\). Write \(A_S\) for its directed Floer category
and \(R_S\) for its total algebra. The paths give affine
coordinates
\[
q_i=\int_{\delta_i}\eta
\]
with the root placed at zero. To compare with a relative
polygonal datum \(\mathfrak d\), we require
\(0\notin\operatorname{Conv}\{q_1,\ldots,q_N\}\), the standing
genericity hypotheses, and compatible path and order data.
The finite-model conditions of Section~\ref{sec:mc-fs-torus}
also require that the relevant moduli spaces and their
degenerations close on the selected labels and path sectors.

Base topology enters this comparison in two ways.
Periods determine the affine input, whereas transport in
\(E\setminus A\) determines additional geometric coefficient
data. A small loop around a critical value has zero
\(\eta\)-period but can have nontrivial Picard--Lefschetz
monodromy. Consequently the periods alone do not specify
the vanishing cycles or their Floer complexes.

The proposed coefficient attached to an edge is a soliton
complex in its prescribed path sector. Subject to the
required analytic and grading hypotheses, it can be modeled
by a Floer complex between vanishing cycles transported to
a common regular fiber along that edge. Rooted summands
assemble chains of these coefficients. Instanton counts
with cyclic soliton asymptotics are expected to supply a
Maurer--Cartan element \(\gamma\), whose image under
\(\Psi_{\mathfrak d}\) deforms the rooted algebra.

In this description, affine polygons in the base index
the algebraic data and the soliton asymptotics. Floer
polygons are solutions with Lagrangian boundary conditions
in a regular fiber or in the total space, according to
the chosen model. The proposed correspondence relates
their counts through degeneration and gluing. Establishing
this relation requires additional analysis and is part of
the following conjecture.

\begin{conjecture}[Conjecture~\ref{conj:fs-total-algebra-from-aoi}]
For \(f\), \(S\), and \(\mathfrak d\) satisfying the admissibility,
affine compatibility, and finite-model conditions of
Section~\ref{sec:mc-fs-torus}, choose compatible soliton coefficients,
gradings, and orientations. The intrinsic instanton model supplies
a Maurer--Cartan element
\[
\gamma_{S,\mathfrak d}\in\mathfrak g_{\mathcal N}^1
\]
and a strictly unital \(A_\infty\)-quasi-isomorphism preserving
the object idempotents,
\[
R^{+,\gamma_{S,\mathfrak d}}_{\mathcal N,\mathfrak d}
\longrightarrow R_S.
\]
When completion is needed, both sides use the same filtered
coefficient framework.
\end{conjecture}

\subsection{Related perspectives and further directions}
The contribution of this paper is an organization
of the secondary-polytope construction over an elliptic
curve, an explicit dependence on relative homotopy data,
and a proposed realization in a curve-valued Lefschetz
fibration. 

Our interest in elliptic bases is also motivated by the
Fukaya category of the elliptic curve itself. The categorical
mirror correspondence of Polishchuk--Zaslow
\cite{PolishchukZaslow} demonstrates the rich algebraic
structures encoded by the symplectic geometry of a torus.
More generally, surface Floer theory admits a concrete
polygonal description: rigid holomorphic discs with pairwise
distinct Lagrangian boundary conditions are represented by
immersed polygons with convex corners
\cite[Remark~2.4]{AurouxSmithSurfaces}.
These descriptions motivate the search for a combinatorial
approach to the Fukaya--Seidel theory of a Lefschetz fibration
\(f:X\to E\), using polygonal geometry on its base.

Our aim is a combinatorial description that combines this
polygonal geometry with Floer-theoretic coefficients and
higher operations supplied by the fibration. In the proposed
geometric realization, Floer complexes of transported vanishing
cycles supply the coefficients, and instanton counts in \(X\)
determine a Maurer--Cartan deformation. The constructions of this paper
develop the algebraic framework for such a description;
the geometric comparison is formulated for compatible
finite collections of thimbles.

A further question is how the algebras attached to different
admissible path systems are related by transport and mutation.
The perverse-schober perspective recalled above suggests a
framework for organizing the local categories and their
monodromy \cite{KSS}. Relating such categorical transport
to the dependence of the polygonal algebras on relative
homotopy data remains a direction for further study.

The elliptic restriction concerns the nonsingular
translation structure. Every smooth complex curve admits
a compatible area form, so symplecticity of the base does
not restrict its genus. On a compact curve of genus \(g>1\),
however, a nonzero holomorphic one-form has zeros of total
order \(2g-2\), and gives translation coordinates only away
from them \cite{ZorichFlatSurfaces}. Although its pullback
to the simply connected universal cover has a primitive,
that primitive has critical points above the zeros and
need not be injective. The hyperbolic universal cover
\cite{FarkasKra} therefore does not by itself reproduce the
affine development used here. Extensions to such curves,
or to punctured curves equipped with suitable differentials,
would require additional control of the polygonal geometry,
factorization, and Floer theory.

\subsection{Outline of this paper}
Section~\ref{sec:secondary-factorization} recalls marked subdivisions,
secondary polytopes, the factorization property, and the cellular
construction of a commutative dg algebra.
Section~\ref{sec:ainf-linf-koszul} fixes conventions for
\(A_\infty\)- and \(L_\infty\)-algebras, Maurer--Cartan elements,
deformation complexes, and the mixed Koszul-dual description of an
\(L_\infty\)-action.

Section~\ref{sec:linf-elliptic} constructs the polygonal
\(L_\infty\)-algebra, its geometric subalgebra, and the comparison
with lifted configurations.
Section~\ref{sec:ainf-elliptic-notation} introduces the exterior root,
constructs the relative algebra and the mixed action, and studies
their dependence on the relative datum and subdivision chambers.

Section~\ref{coeffecient-system} develops coefficient systems in
complexes and bimodules, including the cyclic and rooted tensor
products used in the operations.

Section~\ref{analysis} describes the triangular dg algebra and the
Taylor components of the action explicitly in terms of polygonal
subdivisions.
Section~\ref{uni-thm} proves elliptic universality for the
positive-area directed deformation complex with scalar vertex
algebras. It also gives the Maurer--Cartan deformation formula
for extended coefficients and explains the planar comparison.

Finally, Section~\ref{sec:mc-fs-torus} discusses Fukaya--Seidel
categories of curve-valued Lefschetz fibrations. It formulates the
elliptic complex Morse model, proves the formal deformation
statement assuming a Maurer--Cartan element, and states the
conjectural comparison with the directed Floer algebra of a
compatible finite collection of thimbles.

\begingroup
\makeatletter
\let\addcontentsline\@gobblethree
\subsection*{Acknowledgements}
I would like to express my deep gratitude to my advisor, Yan Soibelman,
for suggesting this problem, for many helpful discussions, and for his
guidance throughout this project. I am also grateful to Gabriel Kerr
for useful conversations and comments, especially concerning
Fukaya--Seidel categories and related geometric aspects of the
construction. I used GPT to assist with phrasing in parts of the
manuscript and to identify typographical errors.

\makeatother
\endgroup

\section{Reminder on secondary polytopes}
\label{sec:secondary-factorization}

We recall the planar secondary-polytope construction and the factorization
property used below. General references are \cite{GKZ1994,DRS}; the
genericity needed for factorization is discussed in \cite[Section~5.2]{KSS}.
In the later elliptic construction, the same definitions will be applied
to the affine domain of a polygon mapping to the base curve.

\subsection{Triangulations and secondary polytopes}
Let \(A=\{a_1,\ldots,a_n\}\subset\mathbb R^2\) span the plane, and put
\(Q=\operatorname{Conv}(A)\). We assume that no three points of \(A\) are
collinear. An additional hypothesis will be imposed when we use
factorization.

\begin{definition}
A \emph{marked polytope} is a pair \((Q,A)\) with
\(Q=\operatorname{Conv}(A)\). A \emph{marked subpolytope}
\((Q',A')\subset(Q,A)\) has \(A'\subset A\) and
\(Q'=\operatorname{Conv}(A')\). Unless stated otherwise, the subpolytopes
used in the planar algebraic construction have dimension two.
\end{definition}

\begin{definition}
A \emph{marked subdivision} of \((Q,A)\) is a collection of its
two-dimensional marked subpolytopes
\(\mathcal S=\{(Q_i,A_i)\}_{i\in I}\), called its maximal cells, such that:
\begin{enumerate}
\item \(Q=\bigcup_i Q_i\);
\item distinct cells meet in a common proper face, possibly empty;
\item the markings agree on intersections:
\[
A_i\cap(Q_i\cap Q_j)=A_j\cap(Q_i\cap Q_j).
\]
\end{enumerate}
A \emph{triangulation} is such a subdivision in which every \(A_i\)
consists of the three vertices of a triangle.
\end{definition}

We do not require \(A_i=A\cap Q_i\), or even \(\bigcup_i A_i=A\).
Thus a triangulation may omit points of \(A\). A subdivision may also
have one maximal cell with marking strictly smaller than \(A\).

\begin{definition}\label{def:geometric-subdivision}
A marked subpolytope \((Q',A')\subset(Q,A)\) is \emph{geometric} if
\(A'=A\cap Q'\). A marked subdivision is geometric if all its maximal
cells are geometric relative to \(A\).
\end{definition}

A height function \(\omega:A\to\mathbb R\) determines the lower convex
envelope
\[
h_\omega(x)=\min\left\{
\sum_{a\in A}t_a\omega(a):
t_a\geq0,\ \sum_a t_a=1,\ \sum_a t_a a=x
\right\}.
\]
Its maximal domains of affinity are the cells \(Q_i\), with markings
\[
A_i=\{a\in A\cap Q_i:\omega(a)=h_\omega(a)\}.
\]
A marked subdivision is \emph{regular} if it arises in this way.
The contact markings are part of the data: the function \(h_\omega\)
alone does not record which nonvertex points lie strictly above it.

We write \(\mathcal S'\preceq\mathcal S\) when \(\mathcal S'\)
refines \(\mathcal S\), including its markings: the cells inside
\((Q_i,A_i)\) form a marked subdivision using subsets of \(A_i\).
The trivial subdivision \(\{(Q,A)\}\) is the largest element.
A nontrivial regular subdivision is \emph{coarse} if its only strictly
coarser regular subdivision is the trivial one. Thus coarse subdivisions
are maximal among the proper subdivisions in this order.

\begin{definition}
For a triangulation \(\mathcal T\), its \emph{GKZ vector} is
\[
\phi_{\mathcal T}(a)=
\sum_{\substack{\Delta\in\mathcal T\\a\in\operatorname{Vert}(\Delta)}}
\operatorname{Area}(\Delta),\qquad a\in A.
\]
The \emph{secondary polytope} \(\Sigma(A)\subset\mathbb R^A\) is the
convex hull of these vectors for all regular triangulations.
\end{definition}

Only triangles having \(a\) as a vertex contribute to this coordinate.
In particular, an unused marked point has coordinate zero.

The nonempty faces of \(\Sigma(A)\) correspond to regular marked
subdivisions, with
\[
F_{\mathcal S'}\subseteq F_{\mathcal S}
\quad\Longleftrightarrow\quad
\mathcal S'\preceq\mathcal S.
\]
This correspondence preserves the specified order. Regular triangulations
correspond to vertices, the trivial subdivision to the whole polytope,
and coarse subdivisions to facets. The dimension of a face is not
determined by the number of maximal cells. For example, either
triangulation of a convex quadrilateral has two maximal cells and
corresponds to a vertex of its secondary interval.

\begin{remark}[Fiber-polytope description]
\label{rem:fiber-polytope-secondary}
For an affinely spanning configuration \(A\subset\mathbb R^d\), let
\[
\Delta_A=\{(t_a)\in\mathbb R^A:t_a\geq0,\ \sum_a t_a=1\},
\qquad
\pi:\Delta_A\to Q,\quad(t_a)\mapsto\sum_a t_a a.
\]
The fiber polytope of \(\pi\) is described by integrals of its sections;
see \cite{BS}. With the unnormalized volume measure,
\[
\left\{\int_Q s(x)\,d\operatorname{Vol}(x):
s:Q\to\Delta_A\text{ a continuous section}\right\}
=\frac{1}{d+1}\Sigma(A).
\]
Indeed, the barycentric section associated to a triangulation satisfies
\[
\int_Q s_{\mathcal T}(x)\,d\operatorname{Vol}(x)
=\frac{1}{d+1}\phi_{\mathcal T},
\]
since each barycentric coordinate integrates to
\(\operatorname{Vol}(\Delta)/(d+1)\) on a simplex \(\Delta\).
The convention using average integrals introduces a further factor
\(1/\operatorname{Vol}(Q)\).
\end{remark}

\begin{proposition}\label{prop:dimension-secondary-polytope}
If \(A\subset\mathbb R^d\) affinely spans \(\mathbb R^d\), then
\[
\dim\Sigma(A)=|A|-d-1.
\]
\end{proposition}

\begin{proof}
Every GKZ vector satisfies
\[
\sum_a\phi_{\mathcal T}(a)=(d+1)\operatorname{Vol}(Q),\qquad
\sum_a\phi_{\mathcal T}(a)a=(d+1)\int_Q x\,d\operatorname{Vol}(x).
\]
These are \(d+1\) independent affine equations. The secondary-fan
theorem identifies the normal fan, with the minimizing convention,
modulo affine height functions with a complete pointed fan in
\(\mathbb R^A/\operatorname{Aff}(\mathbb R^d)\); see \cite{GKZ1994,DRS}.
Consequently there are no further affine equations, giving the claim.
\end{proof}

\subsection{Factorization property}
The cellular construction requires independent refinements of the cells
of each regular subdivision. We make this requirement explicit.
For a planar configuration \(B\), write \(\operatorname{Aff}(B)\)
for restrictions of affine-linear functions to \(B\).

\begin{definition}[Standing factorization hypothesis]
\label{def:factorization-genericity}
For every two-dimensional subconfiguration \(B\subseteq A\) and every
regular marked subdivision
\(\mathcal S=\{(Q_i,B_i)\}_{i\in I}\) of
\((\operatorname{Conv}(B),B)\), we require the restriction map
\begin{equation}\label{eq:factorization-restriction}
\mathbb R^B/\operatorname{Aff}(B)
\longrightarrow
\bigoplus_{i\in I}\mathbb R^{B_i}/\operatorname{Aff}(B_i)
\end{equation}
to be surjective. We impose this hypothesis wherever the
factorization construction is used below.
\end{definition}

\begin{proposition}\label{prop:factorization}
Under this hypothesis, the face corresponding to \(\mathcal S\) has
a factorization
\[
F_{\mathcal S}\cong\prod_{i\in I}\Sigma(B_i).
\]
The factorization respects further regular refinements and gives
\begin{equation}\label{eq:secondary-face-dimension}
\dim F_{\mathcal S}=\sum_i(|B_i|-3),\qquad
\operatorname{codim}F_{\mathcal S}
=|B|-3-\sum_i(|B_i|-3).
\end{equation}
\end{proposition}

\begin{proof}
Choose a height function inducing \(\mathcal S\). Given regular
refinements of its cells, choose height functions for those refinements.
Surjectivity of \eqref{eq:factorization-restriction} gives a height
function on \(B\) whose restrictions agree with these local choices
modulo affine functions. A sufficiently small perturbation of the
original height function by this one realizes all the chosen refinements
simultaneously. Conversely, every regular refinement restricts to such
a collection. This identifies the face posets, compatibly with iteration.

For the GKZ realizations, extend vectors in \(\mathbb R^{B_i}\) by zero
to \(\mathbb R^B\). Addition sends the tuple of GKZ vectors of local
triangulations to the GKZ vector of their union. The dual of
\eqref{eq:factorization-restriction} is the injective map
\[
\bigoplus_i\operatorname{Aff}(B_i)^\perp
\longrightarrow\operatorname{Aff}(B)^\perp
\]
given by this addition. These are the direction spaces of the secondary
polytopes. Hence addition identifies the product affinely with
\(F_{\mathcal S}\). The dimension formulas follow.
\end{proof}

\begin{remark}\label{rem:factorization-KSS}
The factorization hypothesis is stronger than the absence of collinear
triples. A sufficient assumption is that every two-dimensional
subconfiguration is in \emph{exceptionally general position} in the
sense of \cite[Definition~5.2.11]{KSS}. Proposition~5.2.14 there gives
the product face structure under this assumption and explicitly corrects
the missing hypothesis in \cite{KKS}. To see that it implies
\eqref{eq:factorization-restriction}, let \(H_{\mathcal S}\subset\mathbb R^B\)
be the heights affine on each \(B_i\). The cone defining \(\mathcal S\)
is relatively open in \(H_{\mathcal S}\), so
\(\dim F_{\mathcal S}=|B|-\dim H_{\mathcal S}\).
The product dimension therefore says precisely that the restriction
map has full rank. We require the hypothesis for subconfigurations
because the algebra uses all of their secondary polytopes.
\end{remark}

\subsection{The cellular chain complex of a secondary polytope}
Fix a field \(\Bbbk\) of characteristic zero. We use cohomological
gradings, with \(V[k]^i=V^{i+k}\). For a nonempty face \(F\) of a
polytope, its orientation line is
\[
\operatorname{or}(F)=H_{\dim F}(F,\partial F;\Bbbk),
\]
regarded initially as a vector space in degree zero. The cellular chains,
placed in nonpositive cohomological degrees, are
\[
C_\bullet(P)=\bigoplus_{\varnothing\ne F\subseteq P}
\operatorname{or}(F)[\dim F].
\]
The cellular boundary has degree \(+1\) and squares to zero.

For each \(A'\subseteq A\) with two-dimensional convex hull, put
\(d_{A'}=\dim\Sigma(A')=|A'|-3\), and define
\begin{equation}\label{eq:Vdef}
V=\bigoplus_{\substack{A'\subseteq A\\\dim\operatorname{Conv}(A')=2}}
V_{A'},\qquad
V_{A'}=\operatorname{or}(\Sigma(A'))[d_{A'}].
\end{equation}
Each summand is a line in degree \(-d_{A'}\). A choice of orientation
gives a generator \(v_{A'}\); this is a relative top-cell generator,
not an absolute homology class of the polytope.

Form the free graded commutative algebra \(S^\bullet(V)\). Its product
satisfies \(x\odot y=(-1)^{|x||y|}y\odot x\). Using factorization of
each facet and its cellular incidence map, set
\begin{equation}\label{eq:dgenerator}
d(v_{A'})=
\sum_{\mathcal S=\{(Q_i,A_i)\}\text{ coarse}}
\varepsilon(\mathcal S)\bigodot_i v_{A_i}.
\end{equation}
Here the signs express the maps of orientation lines in the chosen
generators. Reordering factors gives the same Koszul signs as reordering
the oriented factors of the face. Formula
\eqref{eq:secondary-face-dimension} gives
\(\sum_i d_{A_i}=d_{A'}-1\), so \(d\) has degree \(+1\).
Extend it by the graded Leibniz rule.

The sum includes one-cell subdivisions omitting an interior mark.
For example, if \(a\) is an interior point of \(\operatorname{Conv}(A')\),
then \((\operatorname{Conv}(A'),A'\setminus\{a\})\) is a coarse
regular subdivision and contributes a linear term. Removing these
subdivisions would not give the full cellular differential.

To check \(d^2=0\), apply \(d\) twice to a generator.
Each term first chooses a facet of \(\Sigma(A')\), then a facet in
one of its product factors. Compatibility of factorization identifies
these choices with flags through a codimension-two face. Each such
face occurs through exactly two facets, with opposite incidence signs.
The terms therefore cancel by the cellular identity \(\partial^2=0\).
Since the square of an odd derivation is a derivation, vanishing on the
generators proves vanishing on all of \(S^\bullet(V)\).
Thus \((S^\bullet(V),d)\) is a commutative dg algebra. Its relation to
an \(L_\infty\)-algebra is recalled next and used in
Section~\ref{sec:linf-elliptic}.

\section{\texorpdfstring{$A_\infty$-algebras, $L_\infty$-algebras and Koszul duality}{A-infinity algebras, L-infinity algebras and Koszul duality}}
\label{sec:ainf-linf-koszul}

This section fixes the algebraic conventions used below. We work over
the field \(\Bbbk\) of characteristic zero and use
\(V[k]^i=V^{i+k}\). All structures are uncurved: there are no nullary
operations or constant components of morphisms. Units, when present,
will be specified separately. References for the bar and operadic
formalisms are \cite{KellerAInfinity,LodayVallette}; for their use in
the algebra of the infrared, see \cite[Section~7]{KKS}.

\subsection{\texorpdfstring{\(A_\infty\)-algebras}{A-infinity algebras}}
\label{sec:ainf-algebras}

\begin{definition}\label{def:ainf-algebra}
An \emph{\(A_\infty\)-algebra structure} on a graded vector space \(V\)
consists of maps \(m_n:V^{\otimes n}\to V\) of degree \(2-n\),
for \(n\geq1\), satisfying
\begin{equation}\label{eq:ainf-stasheff}
\sum_{\substack{r+s+t=n\\s\geq1}}
(-1)^{r+st}m_{r+1+t}
(\operatorname{id}^{\otimes r}\otimes m_s\otimes
\operatorname{id}^{\otimes t})=0.
\end{equation}
Tensor products of maps use the graded convention
\((f\otimes g)(x\otimes y)=(-1)^{|g||x|}f(x)\otimes g(y)\).
Thus additional Koszul signs appear when this formula is evaluated on
homogeneous elements.
\end{definition}

The first two identities say that \(m_1^2=0\) and
\[
m_1m_2(x,y)=m_2(m_1x,y)+(-1)^{|x|}m_2(x,m_1y).
\]
The third says that the associator of \(m_2\) is null-homotopic via
\(m_3\). In particular, \(m_2\) induces an associative product on
\(H^\bullet(V,m_1)\).

\begin{example}\label{ex:dg-algebra-as-ainf}
A dg algebra \((V,d,\mu)\) gives an \(A_\infty\)-algebra with
\(m_1=d\), \(m_2=\mu\), and \(m_n=0\) for \(n\geq3\).
Conversely, an \(A_\infty\)-algebra with these higher operations zero
is a dg algebra.
\end{example}

\begin{definition}\label{def:ainf-morphism}
An \emph{\(A_\infty\)-morphism} \(f:V\to W\) consists of maps
\(f_n:V^{\otimes n}\to W\) of degree \(1-n\), satisfying
\begin{equation}\label{eq:ainf-morphism}
\begin{aligned}
&\sum_{\substack{r+s+t=n\\s\geq1}}(-1)^{r+st}
f_{r+1+t}(\operatorname{id}^{\otimes r}\otimes m_s^V
\otimes\operatorname{id}^{\otimes t})\\
&\qquad=
\sum_{\substack{k\geq1,\ i_1+\cdots+i_k=n\\i_j\geq1}}
(-1)^{\epsilon}m_k^W(f_{i_1}\otimes\cdots\otimes f_{i_k}),
\qquad
\epsilon=\sum_{j=1}^k(k-j)(i_j-1).
\end{aligned}
\end{equation}
\end{definition}

In particular, \(f_1\) is a chain map, and \(f_2\) supplies the
homotopy expressing compatibility of \(f_1\) with the binary products.

\begin{definition}\label{def:strict-ainf-morphism}
The morphism \(f\) is \emph{strict} if \(f_n=0\) for \(n\geq2\).
Then \(f_1\) commutes with every operation.
\end{definition}

\begin{definition}\label{def:ainf-quasi-isomorphism}
The morphism \(f\) is a \emph{quasi-isomorphism} if
\(f_1:(V,m_1^V)\to(W,m_1^W)\) is a quasi-isomorphism of complexes.
\end{definition}

Equivalently, the suspended operations are the Taylor components of
a degree-one coderivation \(b\) on the reduced tensor coalgebra
\[
T^c(V[1])=\bigoplus_{n\geq1}V[1]^{\otimes n},
\]
equipped with deconcatenation coproduct. With the suspension signs
corresponding to \eqref{eq:ainf-stasheff}, the identities are exactly
\(b^2=0\). Similarly, an \(A_\infty\)-morphism is a coalgebra map
\(F:T^c(V[1])\to T^c(W[1])\) satisfying \(Fb_V=b_WF\).
We use this coderivation description for deformation complexes.

\subsection{\texorpdfstring{\(L_\infty\)-algebras}{L-infinity algebras}}
\label{sec:linf-algebras}

We specify the suspension convention explicitly so that the
Maurer--Cartan and twisting formulas have fixed signs. Write
\(sx\in\mathfrak g[1]\) for the image of \(x\in\mathfrak g\), so
\(|sx|=|x|-1\).

\begin{definition}\label{def:linf-algebra}
An \emph{\(L_\infty\)-algebra structure} on \(\mathfrak g\) consists of
graded skew-symmetric maps
\[
\ell_n:\wedge^n\mathfrak g\longrightarrow\mathfrak g,
\qquad |\ell_n|=2-n,\qquad n\geq1.
\]
Define graded symmetric maps of degree one by
\[
q_n(sx_1\odot\cdots\odot sx_n)
=(-1)^{\sum_{a=1}^n(n-a)(|x_a|-1)}
s\ell_n(x_1,\ldots,x_n).
\]
Their extensions define a coderivation \(Q\) on the reduced cofree
cocommutative coalgebra
\[
S^c(\mathfrak g[1])=\bigoplus_{n\geq1}S^n(\mathfrak g[1]).
\]
We require \(Q^2=0\). Explicitly, for homogeneous
\(y_a\in\mathfrak g[1]\), this is
\begin{equation}\label{eq:linf-jacobi}
\sum_{i+j=n+1}\ \sum_{\sigma\in\operatorname{Sh}(i,n-i)}
\varepsilon(\sigma;y)
q_j\bigl(q_i(y_{\sigma(1)},\ldots,y_{\sigma(i)}),
y_{\sigma(i+1)},\ldots,y_{\sigma(n)}\bigr)=0.
\end{equation}
Here \(i,j\geq1\), and \(\varepsilon(\sigma;y)\) is the Koszul sign
for the permutation in the graded symmetric algebra. For permutations
in \(\wedge^n\mathfrak g\), the corresponding sign is
\(\operatorname{sgn}(\sigma)\varepsilon(\sigma;x)\).
\end{definition}

The first identities give \(\ell_1^2=0\) and the graded derivation
identity for \(\ell_1\) and \(\ell_2\). The Jacobiator of \(\ell_2\)
is a boundary determined by \(\ell_3\). Thus \(\ell_2\) induces a
graded Lie bracket on \(H^\bullet(\mathfrak g,\ell_1)\).
Formula \eqref{eq:linf-jacobi} fixes all higher signs without requiring
a second expanded formula on the unsuspended space.

\begin{example}\label{ex:dg-lie-as-linf}
A dg Lie algebra \((\mathfrak g,d,[\ ,\ ])\) is an \(L_\infty\)-algebra
with \(\ell_1=d\), \(\ell_2=[\ ,\ ]\), and \(\ell_n=0\) for
\(n\geq3\).
\end{example}

All infinite series below are interpreted either in a nilpotent setting
or with respect to a complete descending filtration satisfying
\[
\ell_n(F^{p_1}\mathfrak g,\ldots,F^{p_n}\mathfrak g)
\subseteq F^{p_1+\cdots+p_n}\mathfrak g,
\qquad \alpha\in F^1\mathfrak g.
\]
Alternatively, a separate local-finiteness argument may make the sums
finite. Finite-dimensionality alone does not provide convergence.

\begin{definition}\label{def:mc-element}
A degree-one element \(\alpha\) is \emph{Maurer--Cartan} if
\[
\sum_{n\geq1}\frac{1}{n!}\ell_n(\alpha,\ldots,\alpha)=0.
\]
For a dg Lie algebra this is
\(d\alpha+\frac12[\alpha,\alpha]=0\).
\end{definition}

Since \(|s\alpha|=0\), the chosen suspension gives
\(q_n((s\alpha)^{\odot n})=s\ell_n(\alpha,\ldots,\alpha)\).
This explains the signs in the displayed equation. Twisting by \(\alpha\)
gives operations
\[
\ell_n^\alpha(x_1,\ldots,x_n)
=\sum_{k\geq0}\frac1{k!}
\ell_{n+k}(\underbrace{\alpha,\ldots,\alpha}_{k},x_1,\ldots,x_n).
\]
They again satisfy the \(L_\infty\)-identities, and
\(\ell_1^\alpha\) is the differential for deformations around
\(\alpha\).

\begin{definition}\label{def:linf-morphism}
An \emph{\(L_\infty\)-morphism} \(F:\mathfrak g\to\mathfrak h\) is
a coalgebra map
\[
\mathcal F:S^c(\mathfrak g[1])\longrightarrow S^c(\mathfrak h[1]),
\qquad \mathcal FQ_{\mathfrak g}=Q_{\mathfrak h}\mathcal F.
\]
Its degree-zero Taylor components \(f_n:S^n(\mathfrak g[1])\to
\mathfrak h[1]\) correspond to skew-symmetric maps
\(F_n:\wedge^n\mathfrak g\to\mathfrak h\) of degree \(1-n\), with
\[
f_n(sx_1\odot\cdots\odot sx_n)
=(-1)^{\sum_{a=1}^n(n-a)(|x_a|-1)}sF_n(x_1,\ldots,x_n).
\]
\end{definition}

In particular, \(F_1\ell_1^{\mathfrak g}=\ell_1^{\mathfrak h}F_1\).

\begin{definition}\label{def:strict-linf-morphism}
The morphism \(F\) is \emph{strict} if \(F_n=0\) for \(n\geq2\).
\end{definition}

\begin{definition}\label{def:linf-quasi-isomorphism}
The morphism \(F\) is a \emph{quasi-isomorphism} if \(F_1\) is a
quasi-isomorphism for the unary differentials.
\end{definition}

Under the nilpotence or filtration assumptions above, and with \(F\)
compatible with the filtrations, the formula
\[
F_*(\alpha)=\sum_{n\geq1}\frac1{n!}F_n(\alpha,\ldots,\alpha)
\]
sends Maurer--Cartan elements to Maurer--Cartan elements. This
forward construction does not require \(F\) to be a quasi-isomorphism.
It is the formal mechanism used later to deform the relative
\(A_\infty\)-algebra.

\subsection{Derived derivation spaces, Hochschild complexes, and Koszul duality}
\label{sec:derived-derivations}

Let \(\mathcal P\) be a dg operad and \(A\) a dg \(\mathcal P\)-algebra,
with structure map \(\mu:\mathcal P\to\operatorname{End}(A)\).

\begin{definition}
A homogeneous map \(\theta:A\to A\) is a \(\mathcal P\)-derivation if
for every homogeneous \(p\in\mathcal P(n)\) it satisfies the graded
Leibniz identity
\[
\theta\mu(p)=(-1)^{|\theta||p|}
\sum_{i=1}^n\mu(p)
(\operatorname{id}^{\otimes(i-1)}\otimes\theta
\otimes\operatorname{id}^{\otimes(n-i)}).
\]
Tensor products of maps include their Koszul signs. Denote the space
of all such maps by \(\operatorname{Der}_{\mathcal P}(A)\).
\end{definition}

\begin{lemma}
This space is a dg Lie algebra with
\[
[\theta,\eta]=\theta\eta-(-1)^{|\theta||\eta|}\eta\theta,
\qquad
d_{\operatorname{Der}}\theta=d_A\theta-(-1)^{|\theta|}\theta d_A.
\]
\end{lemma}

\begin{proof}
The commutator of two maps satisfying the graded Leibniz identity
again satisfies it. Applying the differential to that identity gives
the same identity for \(d_{\operatorname{Der}}\theta\), because
\(d_{\operatorname{End}}\mu(p)=\mu(d_{\mathcal P}p)\).
The remaining dg Lie identities are inherited from the endomorphism
dg algebra of \(A\).
\end{proof}

Ordinary derivation complexes are not homotopy invariant. Moreover,
a general algebra morphism \(A\to A'\) does not canonically induce a
map \(\operatorname{Der}(A)\to\operatorname{Der}(A')\).
In a model category of \(\mathcal P\)-algebras for which the usual
cofibrant-resolution construction applies, a model for derived
derivations is
\[
R\operatorname{Der}_{\mathcal P}(A)
\simeq\operatorname{Der}_{\mathcal P}(\widetilde A),
\qquad \widetilde A\xrightarrow{\sim}A\text{ cofibrant}.
\]
Here \(\simeq\) denotes the homotopy type of the deformation Lie
algebra; it need not be represented by a canonical strict map between
the derivations of two resolutions. We use the explicit bar models
below for associative and Lie algebras.

For an \(A_\infty\)-algebra \(R\) with bar codifferential \(b\), put
\begin{equation}\label{eq:positive-hochschild}
\begin{aligned}
\mathfrak h_R
&=\operatorname{Coder}(T^c(R[1]))
\cong\prod_{m\geq1}\operatorname{Hom}_{\Bbbk}(R^{\otimes m},R)[1-m],\\
d_{\mathfrak h_R}D&=[b,D].
\end{aligned}
\end{equation}
The bracket is the commutator of coderivations, or equivalently the
shifted Gerstenhaber bracket. We denote this positive-arity complex
by \(C^{\geq1}(R,R)[1]\). Its arity-\(m\) component of internal
degree \(q\) has total degree \(q+m-1\).
The product allows infinitely many Taylor components; compositions
are nevertheless finite on each fixed input length. A direct sum may
replace the product when an arity bound or another finiteness condition
has been established.

For an ordinary associative algebra concentrated in degree zero,
this is the complex
\[
\operatorname{Hom}(R,R)\longrightarrow
\operatorname{Hom}(R^{\otimes2},R)\longrightarrow\cdots,
\]
beginning in degree zero. In particular,
\[
H^0(\mathfrak h_R)=\operatorname{Der}(R).
\]
There is no arity-zero term whose differential quotients out inner
derivations. Thus this complex must be distinguished from the full
shifted Hochschild complex. Fixed units, fixed objects, and directed
coefficient systems require the corresponding normalized or relative
subcomplex, as specified in the later sections.

Similarly, for a Lie algebra \(L\), the adjoint deformation complex
is modeled by the positive-arity Chevalley--Eilenberg cochains
\[
\prod_{m\geq1}\operatorname{Hom}_{\Bbbk}(\wedge^m L,L)[1-m],
\]
with its insertion bracket and differential. Writing the factors as
\(\wedge^m L^*\otimes L\) requires finite-dimensionality; the
\(\operatorname{Hom}\) notation does not.

For a local Artin dg commutative algebra \(B\) with maximal ideal
\(\mathfrak m_B\), a perturbation of the structure is represented by
\[
\gamma\in(\mathfrak h_R\otimes\mathfrak m_B)^1,
\qquad
d_{\mathrm{tot}}\gamma+\tfrac12[\gamma,\gamma]=0.
\]
Here degree one is the total degree, and the differential includes
that of \(B\). Equivalently, the perturbed bar differential squares
to zero. Over ordinary Artin bases, degree-zero gauge transformations
describe changes of identification reducing to the identity. The full
derived deformation problem uses the Maurer--Cartan simplicial set,
which also retains higher homotopies; see \cite{Get}.

For a quadratic Koszul operad \(\mathcal P\), the general formulation
uses square-zero coderivations on the cofree conilpotent coalgebra
over its Koszul dual cooperad \(\mathcal P^{\mathrm i}\).
Passing to derivations of a dual algebra additionally requires
duality and completion hypotheses, as well as the operadic suspension.
For the two cases used here, if \(L\) and \(R\) are finite-dimensional
graded vector spaces, these dual algebras are
\[
\widehat S(L^*[-1])=\prod_{n\geq0}S^n(L^*[-1]),\qquad
\widehat T(R^*[-1])=\prod_{m\geq0}(R^*[-1])^{\otimes m}.
\]
An uncurved \(L_\infty\)- or \(A_\infty\)-structure corresponds to
a continuous degree-one square-zero derivation preserving the
augmentation ideal of the respective algebra. For arbitrary vector
spaces we retain the coderivation description, without identifying
the dual of a tensor power with a tensor power of duals.

\begin{proposition}[Mixed Koszul-dual description]\label{koszul-dual}
Let \(L\) and \(R\) be finite-dimensional graded vector spaces. Set
\[
U=L^*[-1],\qquad V=R^*[-1],\qquad
\mathcal A=\widehat S(U)\widehat\otimes\widehat T(V)
=\prod_{n,m\geq0}S^n(U)\otimes V^{\otimes m},
\]
with the total-word-length topology and graded tensor-product
multiplication. The following data are equivalent:
\begin{enumerate}
\item[(i)] An uncurved \(L_\infty\)-structure on \(L\), an uncurved
\(A_\infty\)-structure on \(R\), and an \(L_\infty\)-morphism
\[
\alpha:L\longrightarrow\mathfrak h_R=C^{\geq1}(R,R)[1].
\]
\item[(ii)] A continuous degree-one derivation \(d\) on \(\mathcal A\)
with \(d^2=0\) and
\begin{equation}\label{eq:mixed-preservation}
d(U)\subseteq\widehat S_+(U)\otimes1,\qquad
d(V)\subseteq\widehat S(U)\widehat\otimes\widehat T_+(V).
\end{equation}
\end{enumerate}
The first condition makes \(\widehat S(U)\otimes1\) a dg subalgebra.
The second preserves the genuine closed ideal
\(J=\widehat S(U)\widehat\otimes\widehat T_+(V)\) and excludes
nullary operations with values in \(R\).
\end{proposition}

\begin{proof}
Assume (ii). Restriction gives a square-zero derivation \(d_L\) on
\(\widehat S(U)\), hence the \(L_\infty\)-structure.
The closed ideal \(I=\widehat S_+(U)\widehat\otimes\widehat T(V)\)
is also preserved by \(d\). Therefore the quotient
\(\mathcal A/I\cong\widehat T(V)\) inherits a square-zero derivation
\(d_R\), which gives the \(A_\infty\)-structure on \(R\).

On the generators \(V\), write
\[
d|_V=d_R|_V+\sum_{n,m\geq1}d_{n,m},\qquad
d_{n,m}:V\longrightarrow S^n(U)\otimes V^{\otimes m}.
\]
The first term has symmetric degree zero; all the other terms have
positive symmetric degree. Graded duality gives degree-one maps
\[
q_{n,m}:S^n(L[1])\otimes(R[1])^{\otimes m}\longrightarrow R[1].
\]
For each \(n\), extend in the \(R[1]\)-variables as a coderivation.
Equivalently, these are degree-zero Taylor maps
\(S^n(L[1])\to\mathfrak h_R[1]\). After desuspension they give
\[
\alpha_n:\wedge^n L\longrightarrow\mathfrak h_R,
\qquad |\alpha_n|=1-n.
\]
In particular, the unsuspended arity-\(m\) component is a map
\[
\alpha_{n,m}:\wedge^n L\otimes R^{\otimes m}\longrightarrow R
\quad\text{of degree }2-n-m.
\]
All dualizations and suspensions here use their graded signs.

Let \(D\) be the derivation on \(\mathcal A\) given by the mixed
terms, extended to vanish on \(U\). Then
\(d=d_L+d_R+D\), where \(d_L,d_R\) act on their respective factors.
Since \(d_L^2=d_R^2=0\) and \([d_L,d_R]=0\), the remaining equation is
\[
[d_L+d_R,D]+\tfrac12[D,D]=0.
\]
Upon graded duality, the term with \(d_L\) inserts the brackets of
\(L\); the term with \(d_R\) gives \([b,-]\); and \([D,D]/2\)
gives the insertion bracket of the mixed coderivations. These are
precisely the Taylor equations for an \(L_\infty\)-morphism
\(\alpha:L\to\mathfrak h_R\).

Conversely, from (i), take the dual derivations \(d_L,d_R\) and the
graded duals \(d_{n,m}\) of the mixed Taylor components. Define \(d\)
on \(U\) by \(d_L\), and on \(V\) by
\(d_R+\sum_{n,m\geq1}d_{n,m}\). These assignments extend uniquely
to a continuous derivation of \(\mathcal A\): \(U\) remains graded
central, and the images of generators have positive total word length.
The preservation conditions \eqref{eq:mixed-preservation} hold.
The structure identities and the morphism identities give \(d^2=0\)
by the same decomposition. The constructions are inverse because
they recover the same Taylor components.
\end{proof}

\begin{remark}
The finite-dimensional assumption in this proposition concerns its
description by dual algebras. The equivalent description by
\(q_n\), \(b_m\), and \(q_{n,m}\), subject to the coderivation
identities, applies without this assumption. Polynomial algebras
\(S(U)\otimes T(V)\) suffice when only finitely many terms occur on
each generator. Mixed terms with \(n,m\geq1\) are essential: requiring
\(d(V)\subseteq1\otimes\widehat T_+(V)\) would force the action to
vanish.
\end{remark}

\section{Construction of the $L_\infty$-algebra}
\label{sec:linf-elliptic}

Let \(A\subset E\) be a finite point configuration on an elliptic curve.
The construction in this section uses the translation structure of \(E\)
to define marked affine polygons and their secondary polytopes.
The polygons are regarded as maps to \(E\), so their images are allowed
to overlap. This permits a uniform treatment of both embedded polygonal
regions and configurations with nontrivial winding.

The finite set \(A\) alone does not specify the polygonal data.
We first fix a marked affine polygon over \(E\) and construct its
\(L_\infty\)-algebra. We then describe the additional data in terms of
relative homotopy classes and compare this intrinsic formulation with
the construction using lifted configurations.

\subsection{Marked polygons over an elliptic curve}

Fix a nonzero holomorphic one-form
\[
\eta\in H^0(E,\Omega_E^1).
\]
Since \(\eta\) is nowhere vanishing, its local primitives give coordinates
whose transition maps are translations. They define a flat metric and
an area form
\[
ds_\eta^2=|\eta|^2,
\qquad
dA_\eta=\frac{i}{2}\eta\wedge\overline{\eta}.
\]
We use the orientation induced by the complex structure.
See \cite{ZorichFlatSurfaces} for background on translation structures.

\begin{definition}
\label{def:affine-polygon-elliptic}
A \emph{marked affine polygon over \(E\)} is a triple
\[
\mathcal P=(P,M,\rho)
\]
with the following properties:
\begin{enumerate}
\item
\(P\) is a compact disk with corners, equipped with a translation
structure, and \(\rho:P\to E\) is locally a translation.
At boundary points this condition is understood using local extensions
across the boundary.
\item
A developing map
\[
\operatorname{dev}:P\longrightarrow\mathbb C,
\qquad
d\operatorname{dev}=\rho^*\eta,
\]
is injective and has image a two-dimensional convex polygon.
\item
\(M\subset P\) is a finite set containing all vertices of \(P\), and
\(\rho|_M:M\to A\) is a bijection.
\end{enumerate}
Thus
\[
\operatorname{dev}(P)
=
\operatorname{Conv}\bigl(\operatorname{dev}(M)\bigr).
\]
We identify two such triples if there is a marking-preserving
translation isomorphism of their domains commuting with their maps to \(E\).
\end{definition}

The developing map is obtained by integrating \(\rho^*\eta\) on the
simply connected domain \(P\), and is unique up to translation.
Consequently, convexity, straight segments, and affine subdivisions on
\(P\) are independent of its choice. The map \(\rho\) need not be
injective. All polygonal constructions are performed on the domain
\(P\), retaining \(\rho\) as part of the data.

If \(\rho\) is injective, we may identify \(P\) with its image in \(E\).
This recovers the description by an embedded marked polygonal region.
In general, the subset \(\rho(P)\subset E\) does not determine
\(\mathcal P\).

Throughout this section we assume that \(\operatorname{dev}(M)\) is in
general position: no three of its points are collinear.
This condition is independent of the choice of developing map.

\begin{definition}
\label{def:marked-subpolygon-elliptic}
For a subset \(B\subset M\) whose affine span has dimension two, put
\[
P_B
:=
\operatorname{dev}^{-1}
\left(
\operatorname{Conv}\bigl(\operatorname{dev}(B)\bigr)
\right).
\]
A \emph{marked subpolygon} of \(\mathcal P\) is the triple
\[
\mathcal P_B=(P_B,B,\rho|_{P_B}).
\]
It is called \emph{geometric} if
\[
B=M\cap P_B.
\]
\end{definition}

An arbitrary marked subpolygon need not contain every point of
\(M\cap P_B\) in its marking. This distinction will be needed for the
cellular differential. The marking \(M\) is part of the fixed finite
input; it is not replaced by the possibly larger set
\(\rho^{-1}(A)\cap P\).

\subsection{Secondary polytopes and the cellular construction}

Let \(\mathcal P_B=(P_B,B,\rho|_{P_B})\) be a marked subpolygon.
A triangulation \(\mathcal T\) of \((P_B,B)\) is a decomposition of \(P_B\)
into affine triangles with vertices in \(B\), meeting along common
faces. We do not require every point of \(B\) to occur as a vertex.
For each triangle \(\Delta\subset P_B\), define
\[
\operatorname{Area}_\eta(\Delta)
=
\int_\Delta\rho^*dA_\eta.
\]
In a developing coordinate \(z=x+iy\), this is the ordinary Euclidean
area of \(\operatorname{dev}(\Delta)\). In particular, area is computed
on the domain, even if the images of different triangles overlap in \(E\).

The GKZ vector of \(\mathcal T\) is
\[
\phi_{\mathcal T}
=
\sum_{\Delta\in\mathcal T}
\operatorname{Area}_\eta(\Delta)
\sum_{b\in\operatorname{Vert}(\Delta)}e_b
\in\mathbb R^B,
\]
where \(\{e_b\}_{b\in B}\) is the standard basis of \(\mathbb R^B\).

\begin{definition}
\label{def:secondary-polytope-elliptic}
The \emph{secondary polytope} of \(\mathcal P_B\) is
\[
\Sigma(\mathcal P_B)
=
\operatorname{Conv}
\left\{
\phi_{\mathcal T}
\ \middle|\
\mathcal T\text{ is a triangulation of }(P_B,B)
\right\}
\subset\mathbb R^B.
\]
When the ambient marked polygon is fixed, we abbreviate this to
\(\Sigma(B)\).
\end{definition}

Development identifies this polytope with the ordinary secondary
polytope of \(\operatorname{dev}(B)\).
Changing the developing map translates the configuration and leaves
all GKZ vectors unchanged. Thus \(\Sigma(\mathcal P_B)\) is independent
of the developing map. Since \(P_B\) has dimension two,
\[
\dim\Sigma(B)=|B|-3.
\]

We use regular \emph{marked} subdivisions in the usual sense.
Explicitly, assign real heights \(w_b\) to the points \(b\in B\)
and take the lower faces of the convex hull of
\[
\left\{
\bigl(\operatorname{dev}(b),w_b\bigr)
\ \middle|\
b\in B
\right\}.
\]
Projecting its maximal lower faces defines a regular subdivision
\[
\mathcal S=\{(P_{B_i},B_i)\}_{i\in I}.
\]
Here \(B_i\) consists of the labels of the lifted marked points lying
on the corresponding lower face. In particular, one only has
\[
B_i\subseteq B\cap P_{B_i};
\]
equality is not part of the definition. A point lying above the lower
convex envelope is omitted from the cell markings.

\begin{proposition}
\label{prop:factorization-elliptic-intrinsic}
For a regular marked subdivision
\(\mathcal S=\{(P_{B_i},B_i)\}_{i\in I}\) of \((P_B,B)\), the corresponding
face of \(\Sigma(B)\) has the factorization
\[
F_{\mathcal S}\simeq\prod_{i\in I}\Sigma(B_i).
\]
These factorizations are compatible with further regular subdivision
of the cells.
\end{proposition}

\begin{proof}
A developing map identifies the marked polygons and their regular
subdivisions with planar ones. The assertion is therefore the usual
factorization property of secondary polytopes. The general-position
hypothesis ensures that a common edge has no additional marked points,
so no nontrivial compatibility factor is contributed by shared edges.
The identification is independent of translation of the development.
\end{proof}

A regular marked subdivision is called \emph{coarse} if its corresponding
face has codimension one. Coarse subdivisions include one-cell
subdivisions obtained by omitting an interior marked point.

We now apply the cellular construction.
Write \(\mathrm{or}(\Sigma(B))\) for the orientation line of the
secondary polytope, and set
\begin{equation}
\label{eq:V-elliptic}
V
:=
\bigoplus_{\substack{B\subset M\\ \dim P_B=2}}V_B,
\qquad
V_B
:=
\mathrm{or}(\Sigma(B))[\dim\Sigma(B)].
\end{equation}
Thus \(V_B\) is concentrated in degree \(-\dim\Sigma(B)\).

For a coarse subdivision
\(\mathcal S=\{(P_{B_i},B_i)\}_{i\in I}\) of \((P_B,B)\),
the cellular incidence map to \(F_{\mathcal S}\), followed by
factorization, gives a degree \(+1\) map
\[
d_{\mathcal S}:V_B\longrightarrow
\bigodot_{i\in I}V_{B_i}\subset S^\bullet(V).
\]
Define \(d|_{V_B}=\sum_{\mathcal S}d_{\mathcal S}\) and extend \(d\)
to \(S^\bullet(V)\) as a graded derivation.
After choosing oriented generators \(v_B\in V_B\), this becomes
\begin{equation}
\label{eq:dv-elliptic}
d(v_B)
=
\sum_{\mathcal S}
\varepsilon(\mathcal S)
\bigodot_{i\in I(\mathcal S)}v_{B_i},
\end{equation}
where the sum runs over coarse regular marked subdivisions and
\(\varepsilon(\mathcal S)\) is the cellular incidence sign with the
product-orientation convention. Permuting the factors gives the
corresponding Koszul sign.

\begin{proposition}
\label{prop:d2zero-elliptic}
The derivation \(d\) satisfies \(d^2=0\). Hence
\((S^\bullet(V),d)\) is a commutative dg algebra.
\end{proposition}

\begin{proof}
Since \(d\) has odd degree, its square is a derivation, so it suffices
to check the assertion on each \(V_B\).
The terms in \(d^2|_{V_B}\) correspond to two successive cellular
boundary maps in \(\Sigma(B)\).
Compatibility of factorization with refinement identifies these terms
with the Leibniz expansion of \eqref{eq:dv-elliptic}.
They cancel by the cellular identity \(\partial^2=0\).
\end{proof}

The derivation has no constant term and preserves the augmentation
ideal. By the usual Koszul-duality construction, it defines an
\(L_\infty\)-structure on
\begin{equation}
\label{eq:g-elliptic}
\overset{\bullet}{\mathfrak g}_{\mathcal P}
:=
V^*[-1]
=
\bigoplus_{\substack{B\subset M\\ \dim P_B=2}}E_B,
\qquad
E_B
:=
V_B^*[-1]
=
\mathrm{or}(\Sigma(B))^*[-\dim\Sigma(B)-1].
\end{equation}
All sums here are finite. The summand \(E_B\) has degree
\(\dim\Sigma(B)+1=|B|-2\), and the brackets \(\ell_k\) have degree \(2-k\).

More explicitly, a component
\[
\ell_k:
E_{B_1}\otimes\cdots\otimes E_{B_k}\longrightarrow E_B
\]
can be nonzero only if the marked polygons
\((P_{B_1},B_1),\ldots,(P_{B_k},B_k)\) form a coarse regular marked
subdivision of \((P_B,B)\). Its sign and suspension convention are
those obtained by dualizing the corresponding component of \(d\).

\medskip

Define
\[
\mathfrak g_{\mathcal P}
:=
\bigoplus_{\substack{B\subset M,\ \dim P_B=2\\ B=M\cap P_B}}E_B
\subset
\overset{\bullet}{\mathfrak g}_{\mathcal P}.
\]
When the ambient data are understood, we write
\(\overset{\bullet}{\mathfrak g}\) and \(\mathfrak g\).
Using the bijection \(M\to A\), these summands may also be indexed by
the corresponding subsets \(A'\subset A\), as \(V_{A'}\) and \(E_{A'}\).

\begin{proposition}
\label{prop:geom-subalg}
The space \(\mathfrak g_{\mathcal P}\) is closed under all brackets
\(\ell_k\). It is therefore an \(L_\infty\)-subalgebra of
\(\overset{\bullet}{\mathfrak g}_{\mathcal P}\).
Its unary bracket is zero.
\end{proposition}

\begin{proof}
Consider a nonzero component of \(\ell_k\) with inputs in
\(E_{B_1},\ldots,E_{B_k}\) and output in \(E_B\).
It is dual to a term
\begin{equation}
\label{eq:dgeom-monomial}
\pm v_{B_1}\odot\cdots\odot v_{B_k}
\end{equation}
in \(d(v_B)\), and hence the input polygons form a subdivision of
\(P_B\). In particular,
\[
P_B=\bigcup_{i=1}^kP_{B_i},
\qquad
B_i\subset B\subset M\cap P_B.
\]
Suppose all inputs are geometric, so that \(B_i=M\cap P_{B_i}\).
For any \(m\in M\cap P_B\), choose a cell \(P_{B_i}\) containing \(m\).
Then \(m\in B_i\subset B\). It follows that
\[
M\cap P_B\subset B,
\]
and therefore \(B=M\cap P_B\). The output summand is geometric,
which proves closure under every bracket.

For \(k=1\), a nonzero component would require a proper one-cell
subdivision with \(P_{B_1}=P_B\) and \(B_1\subsetneq B\).
If its input were geometric, then
\[
B_1=M\cap P_{B_1}=M\cap P_B,
\]
which is incompatible with
\(B_1\subsetneq B\subset M\cap P_B\).
Thus \(\ell_1\) vanishes on \(\mathfrak g_{\mathcal P}\).
\end{proof}

Notice that this proof concerns geometric \emph{inputs} of the dual
brackets. A regular subdivision of a geometric marked polygon need
not have geometric cells: it can omit an interior marked point.
Accordingly, closure of \(\mathfrak g_{\mathcal P}\) should not be
deduced from an assertion that \(d\) preserves the geometric summands
of \(V\).

\subsection{Homotopy data and comparison with lifted configurations}
\label{different-lift}

We now describe the data underlying \(\mathcal P\) without choosing a
polygonal domain in advance. Let
\[
A=\{a_1,\ldots,a_r\}\subset E.
\]
A \emph{coherent system of relative homotopy classes} is a collection
\[
h_{ij}\in\Pi_1(E;a_i,a_j)
\]
satisfying
\[
h_{ii}=1,\qquad
h_{ji}=h_{ij}^{-1},\qquad
h_{ij}*h_{jk}=h_{ik}.
\]
Here \(\Pi_1(E;a_i,a_j)\) denotes homotopy classes of paths in \(E\)
with fixed endpoints, and \(*\) means concatenation from left to right.
These are classes in the unpunctured curve.

A marked affine polygon \(\mathcal P=(P,M,\rho)\) determines such a
system: if \(m_i\in M\) maps to \(a_i\), project any path in \(P\)
from \(m_i\) to \(m_j\). Since \(P\) is a disk, the resulting class
\(h_{ij}\) is independent of the chosen path.

Conversely, a coherent system determines complex displacements
\[
Z_{ij}:=\int_{h_{ij}}\eta,
\qquad
Z_{ij}+Z_{jk}=Z_{ik}.
\]
Fixing \(a_1\) as a reference point, set \(z_i=Z_{1i}\).
Changing the reference point translates all the \(z_i\) by the same
constant.

To see that these data determine a marked affine polygon over \(E\),
let
\[
\Lambda_\eta
=
\left\{\int_\gamma\eta\ \middle|\ \gamma\in H_1(E,\mathbb Z)\right\}
\subset\mathbb C
\]
be the period lattice. Integration identifies \(E\), with \(a_1\)
as origin, with \(\mathbb C/\Lambda_\eta\).
Assuming that \(\{z_i\}\) spans a two-dimensional affine configuration
in general position, take
\[
P=\operatorname{Conv}\{z_1,\ldots,z_r\},
\qquad
M=\{z_1,\ldots,z_r\},
\]
and let \(\rho:P\to E\) be the restriction of the resulting quotient
map. Its induced relative homotopy classes are precisely \(h_{ij}\).
Changing the reference point gives an isomorphic marked affine polygon.
We therefore also denote its geometric algebra by
\[
\mathfrak g_{A,h}.
\]
This notation records the relative homotopy data in addition to \(A\).

For comparison with the covering-space description, choose a universal
cover
\[
\pi:\mathbb C\longrightarrow E,
\qquad
\pi^*\eta=dz.
\]
A choice of one lift \(\widetilde a_1\) of \(a_1\) and the classes
\(h_{1i}\) determines unique lifts \(\widetilde a_i\) of the other
marked points. They satisfy
\[
\widetilde a_j-\widetilde a_i=Z_{ij}.
\]
Changing \(\widetilde a_1\) translates the entire configuration by
one element of \(\Lambda_\eta\). Conversely, a choice of one lift of
each marked point determines a coherent system by projecting paths
between the lifts. Thus coherent systems correspond to lifted
configurations modulo simultaneous deck translation.

\begin{proposition}
\label{prop:intrinsic-lift-comparison}
Let \(\widetilde A\subset\mathbb C\) represent the coherent system \(h\).
Then affine development gives natural identifications
\[
\overset{\bullet}{\mathfrak g}_{\mathcal P}
\cong
\overset{\bullet}{\mathfrak g}_{\widetilde A},
\qquad
\mathfrak g_{A,h}\cong\mathfrak g_{\widetilde A},
\]
with the ordinary planar secondary-polytopal algebras.
These identifications are unchanged by simultaneous deck translation.
\end{proposition}

\begin{proof}
The map from \(P\) to
\(\operatorname{Conv}(\widetilde A)\) is a translation in developing
coordinates. It identifies all marked subpolygons, their GKZ vectors,
and their regular subdivision faces. The identifications preserve
the orientation lines and factorization incidence maps, hence the
derivation \(d\) and its dual brackets. They also preserve the condition
that a marking is geometric.
\end{proof}

In particular, no injectivity assumption on \(\rho\) is needed.
Changing the lift of only some marked points changes the coherent
system \(h\), and may change the polygonal and algebraic data.
The intrinsic formulation removes the choice of an absolute sheet;
it retains the relative homotopy information.

The choice of \(\eta\) fixes an area normalization.
Replacing it by \(c\eta\), with \(c\in\mathbb C^*\), multiplies the
developed coordinates by \(c\) and all GKZ vectors by \(|c|^2\).
The resulting positive dilations identify the secondary polytopes,
their orientation lines, and all incidence maps. Thus the abstract
\(L_\infty\)-algebras are naturally identified under this change.

\begin{remark}
\label{dependence_AS}
For the relation with Fukaya--Seidel categories, let \(p\in E\setminus A\)
be a regular basepoint and choose admissible paths
\[
S=(\delta_1,\ldots,\delta_r),
\qquad
\delta_i:p\longrightarrow a_i,
\]
whose interiors avoid \(A\).
They determine coherent classes in the unpunctured curve by
\[
h_{ij}=[\delta_i^{-1}*\delta_j].
\]
Hence \(S\) supplies the configuration data used above.

This assignment does not recover \(S\) from \(h\).
A loop around another critical value is null-homotopic in \(E\) and
has zero \(\eta\)-period, but its parallel transport can act
nontrivially on vanishing cycles. Accordingly, the relation to a
Fukaya--Seidel model must retain the admissible paths and their
transport data in addition to the polygonal configuration.
The paths are taken on \(E\), and the associated thimbles are
constructed in the original total space of the potential.
\end{remark}

\begin{remark}
The restriction to elliptic curves in this section concerns the
translation structure used in the construction. For a compact
Riemann surface of genus greater than one, uniformization supplies a
canonical hyperbolic metric and geodesic convexity on the universal
cover. In the Klein model this gives a real projective description of
polygonal subdivisions. A higher-genus extension can therefore be
investigated through projective transition maps and their action on
the cellular construction; it is not ruled out by the absence of
translation coordinates. Its relation to a holomorphic
Fukaya--Seidel model requires a separate analysis.
\end{remark}

We now give several low-dimensional examples of the \(L_\infty\)-algebras constructed above.
\begin{example}[Three points]
Let
\[
B=\{a,b,c\}\subset E
\]
be three distinct points, and choose lifts
\[
\widetilde a,\widetilde b,\widetilde c\in \mathbb C
\]
which form a nondegenerate triangle. Assume that there are no additional lifted
points of \(\widetilde A\) inside
\[
\operatorname{Conv}(\widetilde a,\widetilde b,\widetilde c).
\]
Then the secondary polytope \(\Sigma(B)\) is a point. Hence the corresponding
generator
\[
e_{abc}\in E_B\subset \mathfrak g
\]
does not support any nontrivial higher operation.
\end{example}

\begin{example}[Four points in convex position]
Let
\[
B=\{a,b,c,d\}\subset E
\]
and suppose that we can choose lifts
\[
\widetilde a,\widetilde b,\widetilde c,\widetilde d\in \mathbb C
\]
which are in convex position. Then
\[
\operatorname{Conv}(\widetilde B)
\]
is a quadrilateral. A diagonal gives a coarse regular subdivision into two
triangles. For example, the diagonal \((\widetilde a,\widetilde c)\) gives the
two triangular cells
\[
B_1=\{a,b,c\},
\qquad
B_2=\{a,c,d\}.
\]
The corresponding codimension-one face of \(\Sigma(B)\) contributes a term
\[
\pm v_{abc}\odot v_{acd}
\]
to the differential of \(v_{abcd}\). After dualizing, this gives a binary
operation
\[
\ell_2(e_{abc},e_{acd})=\pm e_{abcd}.
\]

\begin{figure}[h]
\centering
\begin{overpic}[width=0.7\textwidth]{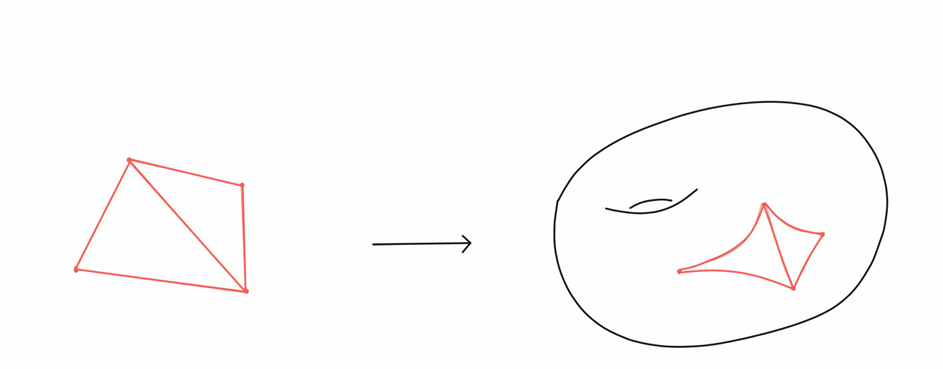}
    \put(15,30){$\mathbb C$}
    \put(75,30){$E$}
    \put(43,15){$\pi$}
    \put(13,23){$\tilde a$}
    \put(26,20){$\tilde b$}
    \put(26.5,6.5){$\tilde c$}
    \put(6,7){$\tilde d$}
\end{overpic}
\caption{Four points in convex position.}
\end{figure}

Similarly, the other diagonal \((\widetilde b,\widetilde d)\) gives
\[
\ell_2(e_{abd},e_{bcd})=\pm e_{abcd}.
\]

\end{example}

\begin{example}[Triangle with one interior point]
Let
\[
B=\{a,b,c,d\}\subset E
\]
and suppose that the chosen lifts satisfy
\[
\widetilde d\in
\operatorname{Int}\operatorname{Conv}(\widetilde a,\widetilde b,\widetilde c).
\]
Thus \(\widetilde a,\widetilde b,\widetilde c\) are the vertices of a triangle,
and \(\widetilde d\) is an interior point. The coarse regular subdivision into
three triangles is
\[
\{\widetilde a,\widetilde b,\widetilde d\},
\qquad
\{\widetilde b,\widetilde c,\widetilde d\},
\qquad
\{\widetilde c,\widetilde a,\widetilde d\}.
\]
Equivalently, writing
\[
B_1=\{a,b,d\},
\qquad
B_2=\{b,c,d\},
\qquad
B_3=\{c,a,d\},
\]
the corresponding face of the secondary polytope contributes a term
\[
\pm v_{abd}\odot v_{bcd}\odot v_{cad}
\]
to the differential of \(v_{abcd}\). Dualizing, we obtain a ternary operation
\[
\ell_3(e_{abd},e_{bcd},e_{cad})
=
\pm e_{abcd}.
\]
\begin{figure}[h]
\centering
\begin{overpic}[width=0.7\textwidth]{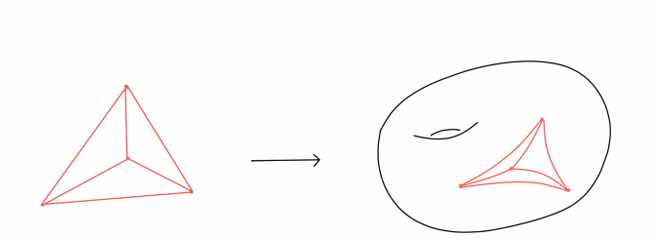}
    \put(15,30){$\mathbb C$}
    \put(75,30){$E$}
    \put(42,14){$\pi$}
    \put(16,23){$\tilde a$}
    \put(20,12){$\tilde d$}
    \put(30,6){$\tilde b$}
    \put(4,4){$\tilde c$}
\end{overpic}
\caption{Triangle with one interior point.}
\end{figure}
\end{example}

\begin{example}[General pattern]
Let
\[
B=\{a_1,\ldots,a_m\}\subset E
\]
and choose lifts
\[
\widetilde B=\{\widetilde a_1,\ldots,\widetilde a_m\}\subset \mathbb C.
\]
Suppose that \(\operatorname{Conv}(\widetilde B)\) admits a coarse regular
subdivision into \(k\) maximal cells
\[
B_1,\ldots,B_k.
\]
Then the corresponding codimension-one face of the secondary polytope
\(\Sigma(B)\) factors as
\[
\prod_{i=1}^k \Sigma(B_i).
\]
This contributes a term
\[
\pm v_{B_1}\odot\cdots\odot v_{B_k}
\]
to the differential of \(v_B\). After dualizing, this gives a potentially
nonzero \(k\)-ary bracket
\[
\ell_k(e_{B_1},\ldots,e_{B_k})
=
\pm e_B.
\]
\end{example}

\medskip
Recall \cite{Get} that an \(L_\infty\)-algebra \(L\) is called nilpotent if
there exists \(r_0>0\) such that all \(r\)-ary iterated superpositions of the
higher brackets vanish identically as maps
\[
        L^{\otimes r}\longrightarrow L
\]
for all \(r>r_0\).

\begin{proposition}
The \(L_\infty\)-algebra \(\mathfrak g_{\widetilde A}\) is nilpotent.
\end{proposition}

\begin{proof}
By construction, the matrix coefficients of the operation \(\lambda_n\) are
indexed by coarse subdivisions of marked subpolytopes \((\widetilde Q',\widetilde A')\subset
(\widetilde Q,\widetilde A)\) into \(n\) marked subpolytopes.  
Therefore the matrix coefficients of an
\(r\)-ary iterated superposition of the operations \(\lambda_n\) are indexed by
subdivisions of marked subpolytopes of \((\widetilde Q,\widetilde A)\) into
\(r\) marked subpolytopes, not necessarily coarse.

Every nonzero iterated operation with $r$ inputs determines a
subdivision of its output polygon into $r$ two-dimensional cells.
Let $b$ be the number of marked points in that polygon and let $h$
be its number of boundary vertices. Refining the cells to a
triangulation using all these points gives $2b-h-2$ triangles.
Consequently,
\[
  r\leq 2b-h-2\leq 2|M|-5.
\]
Thus all iterated operations with $r\geq 2|M|-4$ inputs vanish,
which proves nilpotence.
\end{proof}

\section{Relative setting}
\label{sec:ainf-elliptic-notation}

We now add a marked point \(p\in E\setminus A\), called the
\emph{stop}\footnote{
The terminology is reminiscent of the role of stops in partially wrapped Fukaya
categories; see \cite{Sylvan,GPS}. In the present paper, the stop is used as an
ordering datum in the relative secondary-polytope construction.
}, together with relative homotopy data connecting it to
the configuration. Here the stop supplies an ordering for the
relative polygonal construction; no partially wrapped structure is
assumed. The construction refines the rooted part of an augmented
\(L_\infty\)-algebra to an \(A_\infty\)-algebra. The original, unrooted
\(L_\infty\)-algebra acts on this relative algebra through its Hochschild
complex.

We retain the translation structure, grading conventions, and
orientation lines of Section~\ref{sec:linf-elliptic}. In particular,
all polygons and subdivisions are taken on their affine domains,
even when their images in \(E\) overlap.

\subsection{Relative polygonal data}

\begin{definition}
A \emph{relative polygonal datum} is a marked affine polygon
\[
\mathcal P^+=(P^+,M^+,\rho),
\qquad
M^+=M\sqcup\{r\},
\]
over \(A\sqcup\{p\}\), with \(\rho(r)=p\), such that \(r\) is a vertex
of \(P^+\). We require \(M\) to span a two-dimensional affine
configuration and \(M^+\) to be in general position.
We write
\[
\mathfrak d=(\mathcal P^+,r),
\qquad
\mathcal P^\circ
=(P_{M},M,\rho|_{P_{M}})
\]
for the relative datum and its unrooted polygon.
\end{definition}

Equivalently, in any developing coordinate,
\[
\operatorname{dev}(r)
\notin
\operatorname{Conv}\bigl(\operatorname{dev}(M)\bigr).
\]
This exterior-root condition is the precise hypothesis needed below;
no distance from a fundamental domain is involved.

The same datum can be described by classes
\[
\kappa_i\in\Pi_1(E;p,a_i),
\qquad
h_{ij}=\kappa_i^{-1}*\kappa_j,
\]
where \(h\) is the coherent system of the unrooted configuration.
With the normalization \(\operatorname{dev}(r)=0\), put
\[
z_i=\int_{\kappa_i}\eta.
\]
The conditions are that \(\{0,z_1,\ldots,z_n\}\) be in general
position and that
\[
0\notin\operatorname{Conv}\{z_1,\ldots,z_n\}.
\]
As in Section~\ref{different-lift}, integration then constructs
\(\mathcal P^+\) over \(E\).
Thus the relative datum includes the classes \(\kappa_i\);
the point \(p\) alone does not determine it.

Since the root is exterior, all rays from \(r\) to points of
\(M\) lie in an angular sector of width less than \(\pi\).
The orientation of the translation structure therefore gives an
unambiguous counterclockwise linear order
\[
m_1<\cdots<m_n
\]
on \(M\). General position ensures that no two of these rays
coincide. Restricting this order gives the order for every rooted
subpolygon.

A marked subpolygon \((P_B,B)\) of \(\mathcal P^+\) is
\emph{rooted} if \(r\in B\), and \emph{unrooted} otherwise.
The exterior-root condition implies that a rooted subpolygon has
\(r\) as a vertex, whereas an unrooted subpolygon does not contain
\(r\).

Let
\[
\mathfrak g^+=\mathfrak g_{\mathcal P^+},
\qquad
\mathfrak g=\mathfrak g_{\mathcal P^\circ}.
\]
The geometric summands give a decomposition
\[
\mathfrak g^+=\mathfrak g\oplus\mathfrak g_{\mathrm{root}}.
\]

\begin{proposition}
\label{prop:relative-root-ideal}
Both summands are \(L_\infty\)-subalgebras, and
\(\mathfrak g_{\mathrm{root}}\) is an \(L_\infty\)-ideal in
\(\mathfrak g^+\).
\end{proposition}

\begin{proof}
For a nonzero bracket, the input polygons form a subdivision of the
output polygon. A union of unrooted polygons lies in
\(P_{M}\), so its output is unrooted. If an input contains \(r\),
then the output contains \(r\), which remains an extreme point.
These observations prove both assertions.
\end{proof}

\subsection{The rooted \(A_\infty\)-algebra}

The construction follows the relative two-dimensional subdivision
calculus of \cite[Section~8]{KKS}. Set
\begin{equation}
\label{eq:Vinfty}
V_r
:=
\bigoplus_{\substack{B\subset M^+,\,r\in B\\ \dim P_B=2}}V_B,
\qquad
V_B=\mathrm{or}(\Sigma(B))[\dim\Sigma(B)].
\end{equation}
At this stage the sum includes all marked rooted subpolygons,
including nongeometric ones.

A \emph{rooted subdivision} is a regular marked subdivision all of
whose maximal cells are rooted. The cells have a canonical order:
intersecting them with a sufficiently small circle about \(r\)
gives successive angular intervals, ordered counterclockwise.
Refining a rooted cell into rooted cells replaces its position by
a consecutive ordered block.

For a coarse rooted subdivision
\[
\mathcal S=((P_{B_1},B_1),\ldots,(P_{B_k},B_k))
\]
of \((P_B,B)\), with the cells listed in this order, cellular
incidence and face factorization give a degree \(+1\) map
\[
d_{\mathcal S}:V_B\longrightarrow
V_{B_1}\otimes\cdots\otimes V_{B_k}
\subset T^\bullet(V_r).
\]
Define
\[
d_r|_{V_B}=\sum_{\mathcal S}d_{\mathcal S}
\]
and extend \(d_r\) to the tensor algebra by the graded Leibniz rule.
The signs are the cellular incidence signs with the indicated order
of the orientation-line factors.

\begin{proposition}
\label{prop:d-squares-zero}
One has \(d_r^2=0\). Thus \((T^\bullet(V_r),d_r)\) is an associative
dg algebra.
\end{proposition}

\begin{proof}
The terms in \(d_r^2\) correspond to two successive face inclusions
in a secondary polytope, with every final cell rooted. Every
intermediate cell is a union of final cells, so it is also rooted.
Consequently all intermediate terms in the relevant cellular
boundary cancellation are retained. Compatibility of the angular
order with refinement identifies their tensor orderings.
The terms cancel by \(\partial^2=0\), with the usual product and
Koszul signs.
\end{proof}

Dualizing and suspending defines an uncurved \(A_\infty\)-algebra
\begin{equation}
\label{eq:def-Rinfty}
\overset{\bullet}{R}_{\mathfrak d}
:=
V_r^*[-1]
=
\bigoplus_{\substack{B\subset M^+,\,r\in B\\ \dim P_B=2}}E_B,
\qquad
E_B=\mathrm{or}(\Sigma(B))^*[-\dim\Sigma(B)-1].
\end{equation}
Its operations \(m_k\) have cohomological degree \(2-k\).
The underlying graded space is the rooted part of the augmented
construction, rather than the unrooted space \(\mathfrak g\).
Antisymmetrizing these operations, with the corresponding suspension
signs, recovers the rooted \(L_\infty\)-operations.

Define the geometric rooted subspace by
\begin{equation}
\label{eq:def-Rinfty-geom}
R_{\mathfrak d}
:=
\bigoplus_{\substack{B\subset M^+,\,r\in B,\,\dim P_B=2\\
B=M^+\cap P_B}}E_B
\subset
\overset{\bullet}{R}_{\mathfrak d}.
\end{equation}

\begin{proposition}
\label{prop:Rinfty-subainfty}
The space \(R_{\mathfrak d}\) is an \(A_\infty\)-subalgebra of
\(\overset{\bullet}{R}_{\mathfrak d}\).
\end{proposition}

\begin{proof}
A nonzero component
\[
m_k:E_{B_1}\otimes\cdots\otimes E_{B_k}\longrightarrow E_B
\]
comes from an ordered coarse rooted subdivision of \((P_B,B)\).
Suppose all input markings are geometric. Since
\(P_B=\bigcup_iP_{B_i}\), every point of \(M^+\cap P_B\) lies in some
\(P_{B_i}\), hence in
\(B_i=M^+\cap P_{B_i}\subset B\).
It follows that \(B=M^+\cap P_B\), so the output is geometric.
\end{proof}

\begin{proposition}
\label{prop:Rtp-associative}
The algebra \(R_{\mathfrak d}\) is a graded associative algebra:
\[
m_k|_{R_{\mathfrak d}^{\otimes k}}=0
\qquad (k\neq2).
\]
\end{proposition}

\begin{proof}
Consider a rooted subdivision with \(k\) maximal cells and put
\[
q=\left|B\setminus\bigcup_{i=1}^kB_i\right|.
\]
The cells form a fan at \(r\), separated by \(k-1\) radial diagonals.
General position implies that each common edge has just its two
endpoints as marked points. Counting the repeated root and the
other endpoints of these diagonals gives
\[
\sum_{i=1}^k|B_i|
=
\left|\bigcup_{i=1}^kB_i\right|+2(k-1).
\]
Using \(\dim\Sigma(B)=|B|-3\) and face factorization, we obtain
\[
\operatorname{codim}F_{\mathcal S}
=
(|B|-3)-\sum_{i=1}^k(|B_i|-3)
=
q+k-1.
\]
For a coarse subdivision this equals \(1\). Thus either \(k=2\)
and \(q=0\), or \(k=1\) and \(q=1\).
In particular, \(m_k=0\) for \(k\geq3\), even on the larger algebra.

A unary component has an input \(B_1\subsetneq B\) with
\(P_{B_1}=P_B\). If that input were geometric, then
\[
B_1=M^+\cap P_{B_1}=M^+\cap P_B,
\]
contradicting \(B_1\subsetneq B\subset M^+\cap P_B\).
Hence \(m_1\) vanishes on \(R_{\mathfrak d}\).
There is no nullary operation, and the remaining \(A_\infty\)
relation is associativity of \(m_2\).
\end{proof}

The product has a direct geometric description. For geometric rooted
polygons \((P_{B_1},B_1)\) and \((P_{B_2},B_2)\), the component of
\(m_2\) on \(E_{B_1}\otimes E_{B_2}\) is nonzero precisely when:
\begin{enumerate}
\item their interiors are disjoint and their union is convex;
\item their intersection is a common edge \([r,b]\);
\item \(P_{B_1}\) precedes \(P_{B_2}\) in the counterclockwise
angular order at \(r\).
\end{enumerate}
In this case \(B=B_1\cup B_2=M^+\cap P_B\), and the product is the
incidence map
\[
m_2:E_{B_1}\otimes E_{B_2}\longrightarrow E_{B_1\cup B_2}
\]
with the prescribed suspension signs. Indeed, the two-cell
subdivision is regular: a convex piecewise-affine function with a
single bend along the common edge realizes it. The codimension
calculation above shows that it is coarse.
For chosen oriented generators \(e_B\in E_B\), this reads
\[
m_2(e_{B_1},e_{B_2})=
\begin{cases}
\pm e_{B_1\cup B_2},&\text{if the three conditions hold},\\
0,&\text{otherwise}.
\end{cases}
\]

The angular order also gives a directed decomposition. Let \(R_{ij}\)
be the sum of the geometric rooted summands whose two edges at \(r\)
are \([r,m_i]\) and \([r,m_j]\), with \(i<j\). Then
\[
R_{\mathfrak d}=\bigoplus_{i<j}R_{ij},
\qquad
m_2:R_{ij}\otimes R_{jk}\longrightarrow R_{ik},
\]
and products with nonmatching middle indices vanish.
At this stage \(R_{\mathfrak d}\) is nonunital; diagonal identity
summands may be adjoined when passing to a directed category.

\subsection{The derived derivations}

To retain mixed subdivisions, put
\[
V_0
=
\bigoplus_{\substack{B\subset M\\ \dim P_B=2}}V_B,
\qquad
\mathcal B
=
S^\bullet(V_0)\otimes T^\bullet(V_r).
\]
The factors corresponding to unrooted polygons are graded central. 
For a coarse subdivision of a
rooted polygon, write its unrooted cells as
\(U_1,\ldots,U_n\) and its rooted cells, in angular order, as
\(B_1,\ldots,B_m\). Here \(m\geq1\), since some cell must contain
the extreme vertex \(r\).
The corresponding incidence map takes values in
\[
\bigl(V_{U_1}\odot\cdots\odot V_{U_n}\bigr)
\otimes
V_{B_1}\otimes\cdots\otimes V_{B_m}
\subset\mathcal B.
\]
Unrooted output polygons have only unrooted cells.

Summing these maps over all coarse regular marked subdivisions and
extending by the graded Leibniz rule defines a derivation \(D\) of
\(\mathcal B\). The same cellular boundary argument gives
\[
D^2=0.
\]
Here refinement orders the rooted cells in consecutive blocks,
while reordering unrooted factors contributes the corresponding
Koszul signs. The derivation preserves \(S^\bullet(V_0)\) and
restricts there to the differential encoding the unrooted
\(L_\infty\)-structure. It also preserves the
ideal of positive rooted tensor length,
\[
S^\bullet(V_0)\otimes T^{\geq1}(V_r).
\]
Setting all unrooted generators to zero recovers \(d_r\).

The mixed Koszul-duality construction therefore gives an
\(L_\infty\)-morphism
\[
\overset{\bullet}{\Phi}_{\mathfrak d}:
\overset{\bullet}{\mathfrak g}_{\mathcal P^\circ}
\longrightarrow
C^{\geq1}\bigl(
\overset{\bullet}{R}_{\mathfrak d},
\overset{\bullet}{R}_{\mathfrak d}
\bigr)[1].
\]
More explicitly, the components
\[
D_{n,m}:V_r\longrightarrow
S^n(V_0)\otimes V_r^{\otimes m},
\qquad n\geq1,\quad m\geq1,
\]
dualize, with the standard shifts, to the arity-\(m\) components of
the \(n\)-th Taylor map of
\(\overset{\bullet}{\Phi}_{\mathfrak d}\).
The terms with \(n=0\) define the rooted algebra itself.

Restricting all unrooted and rooted inputs to geometric summands also
forces the output to be geometric. Indeed, these input polygons
cover the output polygon, so every ambient marked point in the
output belongs to an input marking. Consequently we obtain
\[
\Phi_{\mathfrak d}:
\mathfrak g_{\mathcal P^\circ}
\longrightarrow
C^{\geq1}(R_{\mathfrak d},R_{\mathfrak d})[1].
\]
We use the Hochschild dg Lie algebra on the right as the model for
\[
R\mathrm{Der}(R_{\mathfrak d}).
\]
Its bracket is the Gerstenhaber bracket and its differential is
\([m_2,-]\), with the usual graded conventions. Thus
\(\Phi_{\mathfrak d}\) describes the action of the
\(L_\infty\)-algebra by derived derivations of the relative algebra.

\subsection{Dependence on the relative datum and chambers}
\label{A-inf-chamber}

Fix the coherent system \(h\) and the point \(p\).
The compatible systems \(\kappa_i:p\to a_i\) form a torsor for
\(\pi_1(E,p)\): a loop \(\gamma\) acts by
\[
\kappa_i\longmapsto\gamma*\kappa_i
\qquad\text{for every }i.
\]
The classes \(\kappa_i^{-1}*\kappa_j\) remain unchanged.
Under integration this replaces every \(z_i\) by
\(z_i+\int_\gamma\eta\). The exterior-root and general-position
conditions select the admissible relative data in this torsor.

For describing chambers, choose a developing coordinate for the
fixed unrooted polygon and write
\[
x_i=\operatorname{dev}(m_i),\qquad b=\operatorname{dev}(r).
\]
The possible positions \(b\) corresponding to the fixed point \(p\)
form a translate of the period lattice \(\Lambda_\eta\).
We may nevertheless compare their combinatorics by allowing \(b\)
to vary in the auxiliary affine plane. This continuous comparison
does not assert a continuous family of relative homotopy classes
with \(p\) fixed.

Let \(L_{ij}\) be the affine line through \(x_i,x_j\), and set
\[
\mathcal U
=
\left(
\mathbb C\setminus\operatorname{Conv}\{x_1,\ldots,x_n\}
\right)
\setminus\bigcup_{i<j}L_{ij}.
\]
A connected component of \(\mathcal U\) is an \emph{order chamber}.
On such a component, all signs
\[
\operatorname{sgn}\det(x_i-b,x_j-b)
\]
are constant. Since the unrooted configuration is fixed, the entire
oriented order type of the augmented configuration is constant.

\begin{proposition}
\label{indenpend-chamber}
Let \(\mathfrak d,\mathfrak d'\) have the same unrooted datum, with root
positions in the same order chamber.
\begin{enumerate}
\item
After consistent identification of orientation lines, there is a
label-preserving graded algebra isomorphism
\[
I:R_{\mathfrak d}\xrightarrow{\sim}R_{\mathfrak d'}.
\]
\item
Suppose, in addition, that the labeled secondary-polytope face data
for all corresponding marked subconfigurations are identified,
compatibly with refinement, factorization, angular order, and the
oriented incidence maps. Then this identification intertwines the
polygonal actions:
\[
I_*\circ\Phi_{\mathfrak d}=\Phi_{\mathfrak d'},
\]
where \(I_*\) denotes conjugation of Hochschild cochains by \(I\).
\end{enumerate}
\end{proposition}

\begin{proof}
The oriented order type determines convex containment of marked
points, the geometric rooted summands, their angular order, and
whether two such polygons satisfy the three conditions for a
nonzero product. These data are constant in an order chamber.
Transporting the orientation lines consistently therefore identifies
the products, proving the first assertion.

Under the additional hypothesis, all incidence maps defining the
mixed derivation \(D\) agree under the labeled identifications.
The induced morphisms agree after dualization and suspension.
The identification is the identity, since the unrooted datum was
fixed. This proves the second assertion.
\end{proof}

The extra hypothesis in the second assertion should not be replaced
by equality of radial orders. Angular order determines the tensor
orderings, but does not by itself specify which marked subdivisions
are regular. The distinction between oriented order data and
regular subdivision data is also discussed in
\cite[Remark~12.3(b)]{KKS}.
For the pair \((R_{\mathfrak d},\Phi_{\mathfrak d})\), we use the term
\emph{subdivision chamber} for a connected region in an order chamber
equipped with the compatible identifications in the second assertion.

We can also recover the directional description for a distant root.
Fix \(c\in\mathbb C\) and write
\[
b=c+tu,\qquad t>0,\quad u\in S^1.
\]
The asymptotic wall set is
\[
\mathcal W
=
\left\{
\pm\frac{x_i-x_j}{|x_i-x_j|}
\ \middle|\ i\neq j
\right\}\subset S^1.
\]
Indeed,
\[
\det(x_i-b,x_j-b)
=
\det(x_i-c,x_j-c)-t\det(u,x_j-x_i).
\]
For a compact subset \(K\) of a connected component
\(C\subset S^1\setminus\mathcal W\), the nonzero coefficients
\(\det(u,x_j-x_i)\) are bounded away from zero.
There is therefore a threshold \(t_0(K)\) such that all
\(b=c+tu\), with \(u\in K\) and \(t>t_0(K)\), are exterior and have
the same augmented oriented order type. Their geometric rooted
algebras are identified as above.
The distance bound is uniform on compact subsets of \(C\);
a bound uniform up to the walls is not asserted.
This gives a well-defined asymptotic algebra \(R_C\), with consistent
orientation conventions.

For an ordered configuration
\[
S=(a_1,\ldots,a_n)
\]
with its fixed homotopy data, call an order chamber, or an
asymptotic directional chamber, \emph{compatible with \(S\)} if its
angular order induces this ordering. Existence of a compatible
chamber is an additional condition on \(S\).
If only the geometric algebra is needed, an asymptotic compatible
chamber \(C_S\) gives \(R_{C_S}\).
To use both the algebra and its polygonal action, fix a compatible
subdivision chamber \(\mathfrak C_S\) and write
\[
R_{\mathfrak C_S},
\qquad
\Phi_{\mathfrak C_S}:
\mathfrak g_{\mathcal P^\circ}
\longrightarrow R\mathrm{Der}(R_{\mathfrak C_S}).
\]
Their independence of the representative relative datum follows
from Proposition~\ref{indenpend-chamber} with its stated hypotheses.

The period lattice thus indexes choices of relative homotopy data;
it does not automatically act by automorphisms of one fixed
relative algebra. Crossing chamber walls may change the data.
Constructing comparison or mutation maps across such walls requires
additional arguments.

\begin{remark}
If the relative datum comes from admissible Fukaya--Seidel paths
\(\delta_i:p\to a_i\), we set \(\kappa_i=[\delta_i]\) in the
unpunctured curve and check the exterior-root and general-position
conditions above. Compatibility of the angular order with the
chosen order of the paths is also an additional condition.
The paths themselves, together with their transport data, remain
part of the Fukaya--Seidel construction. They lie on \(E\), and their
thimbles lie in the original total space of the potential.
Developing coordinates are used here only to describe the affine
polygonal data and their chambers.
\end{remark}

\section{Algebras with coefficients}
\label{coeffecient-system}

We now equip the polygonal construction with coefficients.
Fix a relative datum
\[
\mathfrak d=(\mathcal P^+,r),
\qquad
\mathcal P^+=(P^+,M^+,\rho),
\qquad
M^+=M\sqcup\{r\},
\]
as in Section~\ref{sec:ainf-elliptic-notation}, and let
\(\mathcal P^\circ\) be its unrooted polygon.
An oriented edge means an oriented affine segment in the polygonal
domain, together with its map to \(E\). Its projected endpoints alone
do not specify its relative homotopy class. For the fixed polygonal
datum, we abbreviate the edge from \(m_i\) to \(m_j\) by \((i,j)\).

\subsection{Coefficient systems and factorizing sheaves}

A \emph{scalar coefficient system} on \(\mathcal P^\circ\) consists
of a bounded finite-dimensional cochain complex \(N_\sigma\) for
each oriented edge between non-root marked points and a perfect
degree-zero chain pairing
\[
\beta_\sigma:
N_\sigma\otimes_\Bbbk N_{\bar\sigma}\longrightarrow\Bbbk.
\]
Here \(\bar\sigma\) denotes the reversed edge. We require
\[
\beta_{\bar\sigma}(y,x)
=
(-1)^{|x||y|}\beta_\sigma(x,y)
\]
for homogeneous elements. Thus the pairings identify
\(N_{\bar\sigma}\) with \(N_\sigma^*\), compatibly with reversal.

For an unrooted marked subpolygon \((P_B,B)\), define
\[
N_B
=
\bigotimes_{\sigma\subset\partial P_B}N_\sigma,
\]
using the induced counterclockwise boundary orientation.
A choice of starting edge presents this tensor product; different
choices are identified by the graded symmetry.

For any regular marked subdivision
\[
\mathcal S=\{(P_{B_\nu},B_\nu)\}_\nu
\]
of \((P_B,B)\), put
\[
N_{\mathcal S}=\bigotimes_\nu N_{B_\nu}.
\]
If \(\mathcal S'\) refines \(\mathcal S\), contract the pairs of
oppositely oriented edges that are internal to the cells of
\(\mathcal S\). This gives a chain map
\[
\gamma_{\mathcal S'\mathcal S}:
N_{\mathcal S'}\longrightarrow N_{\mathcal S}.
\]
Only edges removed in passing from \(\mathcal S'\) to \(\mathcal S\)
are contracted. A refinement that changes the marking without
changing the polygonal boundary contributes the identity on its
boundary coefficient complex.

The degree-zero pairings and graded symmetry imply
\[
\gamma_{\mathcal S''\mathcal S}
=
\gamma_{\mathcal S'\mathcal S}
\circ\gamma_{\mathcal S''\mathcal S'}
\]
whenever \(\mathcal S''\) refines \(\mathcal S'\), which refines
\(\mathcal S\).
Consequently these data define a cellular constructible complex
\(\mathcal N_B\) on \(\Sigma(B)\), with stalk \(N_{\mathcal S}\)
on the open face \(F_{\mathcal S}^{\circ}\).
Since refinement gives
\(F_{\mathcal S'}\subset F_{\mathcal S}\), the maps
\(\gamma_{\mathcal S'\mathcal S}\) are generalization maps from
smaller faces to larger faces.

This construction uses all regular marked subdivisions.
Coarse subdivisions enter when extracting the top cellular
differential. For every face factorization,
\[
F_{\mathcal S}\simeq\prod_\nu\Sigma(B_\nu),
\]
there is a compatible identification
\[
\mathcal N_B|_{F_{\mathcal S}}
\simeq
\boxtimes_\nu\mathcal N_{B_\nu}.
\]
Thus \(\{\mathcal N_B\}\) is a factorizing system of complexes of
sheaves.

To agree with the chain-orientation convention of
Section~\ref{sec:linf-elliptic}, write
\[
d_B=\dim\Sigma(B)=|B|-3,
\qquad
\omega_B=\mathrm{or}(\Sigma(B))^*.
\]
Define
\[
E_B^{\mathcal N}
=
N_B\otimes\omega_B[-d_B-1],
\qquad
\overset{\bullet}{\mathfrak g}_{\mathcal N}
=
\bigoplus_{\substack{B\subset M\\ \dim P_B=2}}
E_B^{\mathcal N}.
\]
Here the internal differential of \(N_B\) is included, with the
usual differential on the shifted complex.

For a coarse regular subdivision of \((P_B,B)\) with cells
\((P_{B_1},B_1),\ldots,(P_{B_k},B_k)\), combine the contraction
\[
\gamma_{\mathcal S}:
N_{B_1}\otimes\cdots\otimes N_{B_k}\longrightarrow N_B
\]
with the dual cellular incidence map on orientation lines.
After suspension this gives a component
\[
\ell_k:
E_{B_1}^{\mathcal N}\otimes\cdots\otimes
E_{B_k}^{\mathcal N}
\longrightarrow E_B^{\mathcal N}
\]
of degree \(2-k\).

Together with the internal coefficient differentials, these maps
define a degree-one coderivation
\[
Q_{\mathcal N}:
S^c\bigl(\overset{\bullet}{\mathfrak g}_{\mathcal N}[1]\bigr)
\longrightarrow
S^c\bigl(\overset{\bullet}{\mathfrak g}_{\mathcal N}[1]\bigr),
\qquad
Q_{\mathcal N}^2=0.
\]
Here \(S^c\) denotes the conilpotent symmetric coalgebra.
The identity follows from the cellular cochain identity, the chain
map property of the contractions, and their compatibility with
refinement. This is the cohomological form of the construction.

Define the geometric subspace
\[
\mathfrak g_{\mathcal N}
=
\bigoplus_{\substack{B\subset M,\,\dim P_B=2\\
B=M\cap P_B}}E_B^{\mathcal N}.
\]
It is an \(L_\infty\)-subalgebra: geometric input polygons cover the
output polygon, so the output marking is geometric, exactly as in
Proposition~\ref{prop:geom-subalg}.
The unary contribution from proper one-cell subdivisions vanishes
on this subalgebra. Its unary bracket is nevertheless generally
nonzero: it is the differential induced by the coefficient complexes.

In particular,
\[
\mathfrak g_{\mathcal N}^{1}
=
\bigoplus_{\substack{B\subset M,\,\dim P_B=2\\
B=M\cap P_B}}
N_B^{-d_B}\otimes\omega_B.
\]
An element of this space assigns degree-zero coefficients to
geometric three-point polygons, degree \(-1\) coefficients to
geometric four-point configurations, and so on.

Such an assignment is a Maurer--Cartan element only when it satisfies
\[
\ell_1(\alpha)
+
\sum_{k\geq2}\frac{1}{k!}
\ell_k(\alpha,\ldots,\alpha)
=0.
\]
For the fixed finite configuration, only finitely many subdivision
arities occur.

\subsection{Bimodule coefficients}

We next replace scalar coefficients by algebras and
bimodules.
The vertex labels in this subsection belong to \(M\), or
equivalently to \(A\) with its fixed polygonal data.

\begin{definition}
\label{def:extended-coeff}
An \emph{extended coefficient system} consists of:
\begin{enumerate}
\item
A unital associative dg algebra \(S_i\) for each unrooted marked point
\(m_i\).
\item
For every oriented edge \((i,j)\), \(i\neq j\), a dg
\((S_i,S_j)\)-bimodule \(N_{ij}\), finitely generated projective over
the graded algebra underlying \(S_i\otimes_\Bbbk S_j^{\mathrm{op}}\).
\item
Degree-zero chain pairings
\[
\beta_{ij}:
N_{ij}\otimes_\Bbbk N_{ji}
\longrightarrow S_i\otimes_\Bbbk S_j
\]
that are morphisms of
\((S_i\otimes_\Bbbk S_j,S_i\otimes_\Bbbk S_j)\)-bimodules.
On the source, the left \(S_i,S_j\) actions are those on
\(N_{ij},N_{ji}\), respectively; the right \(S_i,S_j\) actions are
those on \(N_{ji},N_{ij}\), respectively, with the Koszul signs.
\item
Nondegeneracy: the transpose induced by \(\beta_{ij}\),
\[
N_{ij}
\xrightarrow{\ \sim\ }
\operatorname{Hom}_{S_j\otimes_\Bbbk S_i^{\mathrm{op}}}
\bigl(N_{ji},S_j\otimes_\Bbbk S_i^{\mathrm{op}}\bigr),
\]
is an isomorphism of dg \((S_i,S_j)\)-bimodules.
\item
Graded symmetry:
\[
\beta_{ji}\circ\tau_{N_{ij},N_{ji}}
=
\tau_{S_i,S_j}\circ\beta_{ij},
\]
where \(\tau(x\otimes y)=(-1)^{|x||y|}y\otimes x\).
\end{enumerate}
\end{definition}

\begin{remark}
\label{rem:trivial-coeff}
Taking \(S_i=\Bbbk\) recovers scalar coefficient systems.
Taking in addition \(N_{ij}=\Bbbk\) concentrated in degree zero,
with the standard pairings, recovers the coefficient-free
construction.
\end{remark}

For an oriented boundary chain \(i_0,\ldots,i_m\), define
\begin{equation}
\label{eq:LQprime}
L(i_0,\ldots,i_m)
=
N_{i_0i_1}\otimes_{S_{i_1}}N_{i_1i_2}
\otimes_{S_{i_2}}\cdots
\otimes_{S_{i_{m-1}}}N_{i_{m-1}i_m}.
\end{equation}
This is an \((S_{i_0},S_{i_m})\)-bimodule.
For an unrooted polygon whose boundary vertices are
\(i_0,\ldots,i_m\) in counterclockwise order, set
\begin{equation}
\label{eq:NQprime}
N_B
=
L(i_0,\ldots,i_m)
\otimes_{S_{i_0}^{\mathrm{op}}\otimes_\Bbbk S_{i_m}}
N_{i_mi_0}.
\end{equation}
Here the left \(S_{i_0}\)-action and right \(S_{i_m}\)-action on
\(L\) make it a right module over the displayed tensor algebra;
\(N_{i_mi_0}\) is viewed as a left module over the same algebra.
This cyclic tensor product is independent of the starting vertex,
up to the canonical graded cyclic identifications.

The pairings induce contractions along shared edges.
If two oriented cyclic boundary paths \(\sigma,\tau\) contain
\((i,j)\) and \((j,i)\), respectively, there is a chain map
\begin{equation}
\label{eq:concat-map}
\gamma_{\sigma,\tau}^{[i,j]}:
N_\sigma\otimes_\Bbbk N_\tau
\longrightarrow N_{\sigma*_{[i,j]}\tau}.
\end{equation}
It removes the two edge factors using \(\beta_{ij}\) and inserts
the resulting \(S_i,S_j\) factors into the adjacent bimodules.

For example, write the relevant boundary portions as
\[
\sigma=(\ldots,h,i,j,k,\ldots),
\qquad
\tau=(\ldots,a,j,i,q,\ldots).
\]
If
\[
\beta_{ij}(n_{ij}\otimes n_{ji})
=
\sum_\nu s_\nu'\otimes s_\nu'',
\qquad
s_\nu'\in S_i,\quad s_\nu''\in S_j,
\]
then the contraction is represented by
\begin{equation}
\label{eq:concat-formula}
\sum_\nu(-1)^{\epsilon_\nu}
\bigl(
\cdots\otimes n_{aj}s_\nu''\otimes n_{jk}
\otimes\cdots\otimes n_{hi}s_\nu'\otimes n_{iq}
\otimes\cdots
\bigr).
\end{equation}
The sign \((-1)^{\epsilon_\nu}\) is the Koszul sign for moving the
homogeneous factors into the displayed order before and after
applying the pairing. The bimodule identities ensure that the map
descends to the balanced tensor products.

For a subdivision \(\mathcal S\), these contractions give
\begin{equation}
\label{eq:gammaP}
\gamma_{\mathcal S}:
\bigotimes_\nu N_{B_\nu}\longrightarrow N_B.
\end{equation}
The contractions are compatible with iterated refinement, with all
permutations interpreted in the graded sense. They therefore give
factorizing sheaves \(\mathcal N_B\) and the coefficient
\(L_\infty\)-algebras defined above, now with
\(N_B\) given by \eqref{eq:NQprime}.

For the relative bimodule construction, the root is an ordering
datum. No vertex algebra or edge bimodule at \(r\) is required.
If a rooted polygon has boundary vertices
\[
r,i_0,i_1,\ldots,i_m
\]
in counterclockwise order, use its boundary chain away from the root:
\[
L_B=L(i_0,\ldots,i_m),
\qquad
F_B^{\mathcal N}=L_B\otimes\omega_B[-d_B-1].
\]
Thus coefficients of unrooted polygons are cyclic tensor products, while rooted
coefficients retain their two endpoint module actions.

When all \(S_i=\Bbbk\), this convention agrees with the scalar
construction on the augmented polygon in which every edge incident
to \(r\) is assigned \(\Bbbk\) with its standard pairing.
Allowing nontrivial complexes on those root edges is an additional
choice of scalar relative coefficients.

Set
\[
\overset{\bullet}{R}_{\mathcal N,\mathfrak d}
=
\bigoplus_{\substack{B\subset M^+,\,r\in B\\ \dim P_B=2}}
F_B^{\mathcal N},
\qquad
R_{\mathcal N,\mathfrak d}
=
\bigoplus_{\substack{B\subset M^+,\,r\in B,\,\dim P_B=2\\
B=M^+\cap P_B}}
F_B^{\mathcal N}.
\]
For a mixed subdivision, its coefficient complex is the tensor
product of the cyclic factors \(N_U\) for unrooted cells and the
linear factors \(L_B\) for rooted cells, the latter in angular order.
Refinement maps contract oppositely oriented unrooted edges using
\(\beta_{ij}\). When two rooted cells meet along \([r,m_j]\),
their linear factors are composed through the balanced tensor
product over \(S_j\). These are degree-zero chain maps compatible
with refinement.

Combining these maps with cellular incidence defines a
square-zero coderivation on
\[
T^c\bigl(\overset{\bullet}{R}_{\mathcal N,\mathfrak d}[1]\bigr),
\]
where \(T^c\) is the conilpotent tensor coalgebra.
Geometric inputs again force geometric outputs, so this
\(A_\infty\)-structure restricts to \(R_{\mathcal N,\mathfrak d}\).

\begin{proposition}
\label{prop:coefficient-rooted-dg}
The geometric rooted algebra \(R_{\mathcal N,\mathfrak d}\) is a
nonunital dg associative algebra. Its differential \(m_1\) is induced
by the internal differentials of the complexes \(L_B\), and
\[
m_k=0\qquad(k\geq3).
\]
Its product is given by ordered convex gluing and the corresponding
maps of linear tensor products.
\end{proposition}

\begin{proof}
The classification of coarse all-rooted subdivisions in
Proposition~\ref{prop:Rtp-associative} is unchanged by coefficients.
It excludes operations of arity at least three, and the geometric
input condition excludes the unary subdivision contribution.
The internal coefficient differential remains.
The \(A_\infty\)-identities therefore state that \(m_1^2=0\),
that \(m_1\) is a derivation of \(m_2\), and that \(m_2\) is associative.
\end{proof}

More explicitly, a product can be nonzero only when the two rooted
polygons have disjoint interiors, convex union, a common edge
\([r,m_j]\), and the prescribed angular order. Its coefficient map is
\[
L_{B_1}\otimes_\Bbbk L_{B_2}
\longrightarrow
L_{B_1}\otimes_{S_j}L_{B_2}
\xrightarrow{\sim}L_{B_1\cup B_2},
\]
tensored with the orientation-line incidence map and the suspension
signs. Unlike the coefficient-free case, a permitted gluing need
not give a nonzero map on every choice of coefficients.

If \(R_{ij}\) denotes the sum of rooted summands whose two root
edges end at \(m_i,m_j\), with \(i<j\), then
\[
R_{\mathcal N,\mathfrak d}=\bigoplus_{i<j}R_{ij},
\qquad
m_2:R_{ij}\otimes_{S_j}R_{jk}\longrightarrow R_{ik}.
\]
Putting \(R_{ii}=S_i\) and using the endpoint actions gives the
unital triangular dg algebra
\[
R_{\mathcal N,\mathfrak d}^{+}
=
\bigoplus_{i\leq j}R_{ij}.
\]

The mixed subdivisions also define a polygonal action.
In suspended notation their components have the form
\[
q_{n,m}:
S^n(\mathfrak g_{\mathcal N}[1])
\otimes R_{\mathcal N,\mathfrak d}[1]^{\otimes m}
\longrightarrow R_{\mathcal N,\mathfrak d}[1],
\qquad n\geq1,\quad m\geq1,
\]
and have degree \(+1\).
Each component combines a coefficient contraction with the
corresponding oriented cellular incidence map.

\begin{proposition}
\label{prop:coefficient-polygonal-action}
The mixed components define an \(L_\infty\)-morphism
\[
\Phi_{\mathcal N,\mathfrak d}:
\mathfrak g_{\mathcal N}
\longrightarrow
C^{\geq1}(R_{\mathcal N,\mathfrak d},
         R_{\mathcal N,\mathfrak d})[1]
=
R\mathrm{Der}(R_{\mathcal N,\mathfrak d}).
\]
The differential in the target is
\([m_1+m_2,-]\).
\end{proposition}

\begin{proof}
Apply the mixed cellular construction first to all marked
subpolygons. Compatibility of coefficient contractions with
refinement and the cellular identity in codimension two give the
mixed bar identities. These are the identities for an
\(L_\infty\)-morphism into the Hochschild dg Lie algebra.
Every subdivision of a rooted polygon has at least one rooted
cell, so no nullary Hochschild cochains occur.

Finally, a subdivision with geometric unrooted and rooted inputs has
geometric output, by the covering argument used above.
All the mixed identities therefore restrict to the displayed
geometric spaces.
\end{proof}

The coalgebra formulation avoids any finiteness assumption needed
to dualize the bimodule coefficient spaces over \(\Bbbk\).
When finite dualization is available, the equivalent algebra
formulation uses shifted dual generators
\[
\mathfrak g_{\mathcal N}^*[-1],
\qquad
R_{\mathcal N,\mathfrak d}^*[-1],
\]
in completed symmetric and tensor algebras.

If \(\alpha\in\mathfrak g_{\mathcal N}^{1}\) is Maurer--Cartan,
its image under \(\Phi_{\mathcal N,\mathfrak d}\) deforms the
rooted dg algebra to an \(A_\infty\)-algebra. With the usual
Hochschild suspension convention, its operations are
\[
m_k^\alpha
=
m_k+
\sum_{n\geq1}\frac{1}{n!}
\bigl(
\Phi_{\mathcal N,\mathfrak d,n}
(\alpha,\ldots,\alpha)
\bigr)_k.
\]
Thus higher rooted operations may appear after deformation even
though the undeformed geometric algebra has only \(m_1\) and \(m_2\).

\begin{remark}
    
Keep the unrooted polygon and its coefficient system fixed.
For root positions in a common order chamber, the identified
geometric polygons have the same boundary chains and coefficient
complexes. The ordered gluing maps and internal differentials
therefore identify the rooted dg algebras, with compatible choices
of orientation lines.

Identification of the polygonal actions additionally requires the
compatible regular-subdivision and incidence data specified in
Proposition~\ref{indenpend-chamber}.
More generally, if the coefficient systems themselves vary, their
vertex algebras, edge complexes or bimodules, and pairings must also
be identified compatibly with these contractions.
Under these hypotheses the corresponding
\(\Phi_{\mathcal N,\mathfrak d}\) agree.

\end{remark}

\section{Analysis of the \(A_\infty\)-algebra and the \(L_\infty\)-morphism}
\label{analysis}

Fix a relative datum \(\mathfrak d=(\mathcal P^+,r)\)
and an extended coefficient system \(\mathcal N\) on its unrooted
polygon, as in the preceding sections. Write
\[
M=\{m_1<\cdots<m_N\}
\]
in the angular order at \(r\). All polygons, edges, and subdivisions
in this section are taken on the affine domain \(P^+\).
The figures depict these domains in a developing coordinate.

We abbreviate
\[
\mathfrak g=\mathfrak g_{\mathcal N},
\qquad
R^- =R_{\mathcal N,\mathfrak d},
\qquad
R^+=\bigoplus_{i\leq j}R_{ij},
\qquad R_{ii}=S_i.
\]
Thus \(R^-\) is the nonunital rooted algebra and \(R^+\) is its
triangular unital extension. The geometric marking is
\(B=M^+\cap P_B\) for rooted polygons and \(B=M\cap P_B\)
for unrooted polygons.
We retain the notation
\[
\omega_B=\mathrm{or}(\Sigma(B))^*,
\qquad d_B=\dim\Sigma(B),
\qquad
E_B^{\mathcal N}=N_B\otimes\omega_B[-d_B-1],
\quad
F_B^{\mathcal N}=L_B\otimes\omega_B[-d_B-1].
\]

\subsection{The directed dg algebra}

For \(i<j\), let \(R_{ij}\) be the sum of the rooted summands
whose two edges incident to \(r\) are \([r,m_i]\) and \([r,m_j]\).
Then
\[
R^-=\bigoplus_{i<j}R_{ij}.
\]

\begin{proposition}
\label{prop:rooted-triangular}
The operations on \(R^-\) satisfy
\[
m_k=0\quad(k\geq3),
\qquad
m_1=d_{\mathrm{coeff}}.
\]
The binary product is supported on matching indices and factors as
\[
\mu_{ijk}:R_{ij}\otimes_{S_j}R_{jk}\longrightarrow R_{ik},
\qquad i<j<k.
\]
Consequently \(R^-\) is a strictly upper-triangular dg algebra
without unit. Adjoining \(R_{ii}=S_i\), with the endpoint module
actions, gives the unital triangular dg algebra \(R^+\).
\end{proposition}

\begin{proof}
By Proposition~\ref{prop:coefficient-rooted-dg}, coarse all-rooted
subdivisions contribute only binary operations on geometric
summands. The unary subdivision terms vanish, while the internal
coefficient differential remains. A binary gluing joins consecutive
angular sectors along \([r,m_j]\), and its coefficient map is
balanced over \(S_j\). The dg algebra identities and the assertion
about adjoining diagonal algebras follow.
\end{proof}

In particular, the binary product can be supported only when the
input polygons have disjoint interiors, meet along the indicated
root edge, and have convex union in the prescribed angular order.
The coefficient map may vanish even when these geometric conditions
hold. In the coefficient-free case the internal differential is zero.

\subsection{Taylor components of the polygonal action}

Let
\[
\Phi:\mathfrak g\longrightarrow C^{\geq1}(R^-,R^-)[1]
\]
be the coefficient-enhanced polygonal action.
An \(L_\infty\)-morphism has Taylor components
\(\Phi_n\), \(n\geq1\); its first component alone does not specify
the full morphism.

The arity-\(m\) matrix components of \(\Phi_n\) have the form
\[
\Phi^C_{A_1,\ldots,A_n;B_1,\ldots,B_m}:
E_{A_1}^{\mathcal N}\otimes\cdots\otimes E_{A_n}^{\mathcal N}
\otimes F_{B_1}^{\mathcal N}\otimes\cdots\otimes F_{B_m}^{\mathcal N}
\longrightarrow F_C^{\mathcal N},
\]
of degree \(2-n-m\).
They are obtained from coarse regular marked subdivisions with
\(n\) unrooted cells and \(m\) rooted cells. The rooted cells are
listed in angular order; the unrooted inputs have the symmetry prescribed
by the suspended \(L_\infty\)-convention.

For \(n=1\), write the corresponding component as
\[
\Phi_1^{(B_1,\ldots,B_m\mid C)}|_{E_{A'}^{\mathcal N}}:
E_{A'}^{\mathcal N}\longrightarrow
\operatorname{Hom}_{\Bbbk}
\bigl(F_{B_1}^{\mathcal N}\otimes\cdots\otimes F_{B_m}^{\mathcal N},
F_C^{\mathcal N}\bigr)[1-m].
\]
Equivalently, its unshifted transposed map has degree \(1-m\).
We now describe these components with one unrooted input geometrically,
following the analysis of subdivisions with one finite cell in
\cite[Section~11]{KKS}.

Let \((Q',A')\) be a geometric unrooted subpolygon.
The root lies outside \(Q'\). Define its rooted hull and marking by
\[
\widehat Q'=\operatorname{Conv}_{P^+}\bigl(Q'\cup\{r\}\bigr),
\qquad
\widehat A'=M^+\cap\widehat Q'.
\]
The convex hull is taken in the affine structure of the domain.

The \emph{positive boundary} \(\partial^+Q'\) consists of those
edges whose supporting lines separate the root from the interior
of \(Q'\). These are the edges visible from \(r\).
The complementary boundary arc is the \emph{negative boundary}
\(\partial^-Q'\). General position makes this distinction strict.
The two arcs meet at vertices \(a,b\), listed in angular order
from \(r\).

\begin{figure}[htbp]
\centering
\begin{overpic}[width=0.35\textwidth]{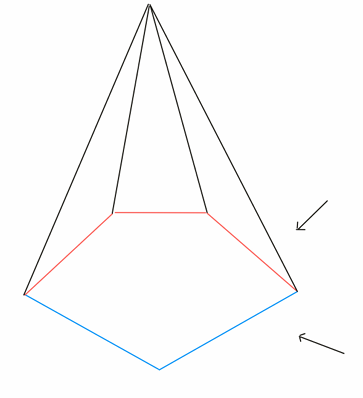}
    \put(32,97){\(r\)}
    \put(37,27){\(Q'\)}
    \put(90,10){\(\partial^-Q'\)}
    \put(85,50){\(\partial^+Q'\)}
    \put(20,50){\(\Pi_1\)}
    \put(53,50){\(\Pi_m\)}
    \put(36,50){\(\cdots\)}
\end{overpic}
\caption{The positive and negative boundary arcs of \(Q'\), shown
in a developing coordinate on the affine domain.}
\label{fig:positive-negative-boundary}
\end{figure}

Write the edges of the positive arc as
\[
\eta_1,\ldots,\eta_m
\]
in angular order at \(r\), and orient them consecutively from
\(a\) to \(b\). These orientations agree with their orientations
as finite edges of the rooted triangles and are opposite to their
boundary orientations in \(Q'\).
For each \(\nu\), set
\[
\Pi_\nu=\operatorname{Conv}_{P^+}(r,\eta_\nu),
\qquad
D_\nu=M^+\cap\Pi_\nu.
\]
Then
\[
(\widehat Q',\widehat A')
=
(Q',A')\cup(\Pi_1,D_1)\cup\cdots\cup(\Pi_m,D_m)
\]
is a regular marked subdivision with one unrooted cell.
A triangular support \(\Pi_\nu\) may contain additional marked
points; its geometric marking is always \(D_\nu\).

More generally, suppose that \((P,C)\) is a geometric rooted polygon,
\(Q'\subset P\), and \(\partial^-Q'\) is a union of edges of the
boundary of \(P\) away from \(r\).
The remainder of that boundary consists of a left handle
\(\lambda\), from the first non-root vertex of \(P\) to \(a\), and
a right handle \(\rho\), from \(b\) to the last non-root vertex.
Either handle may be empty.
The resulting subdivision into \(Q'\) and rooted cells
\(P_1,\ldots,P_m\) is called a \emph{subdivision with one unrooted cell}
(or a \(1\)-finite subdivision).

\begin{figure}[htbp]
\centering
\begin{overpic}[width=0.6\textwidth]{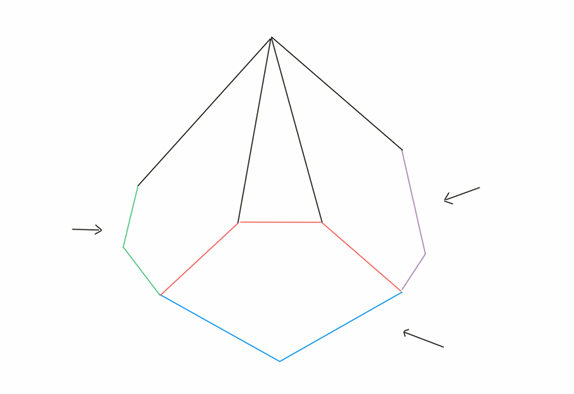}
    \put(45,64){\(r\)}
    \put(46,20){\(Q'\)}
    \put(80,6){\(\partial^-Q'\)}
    \put(85,35){Right handle \(\rho\)}
    \put(-15,27){Left handle \(\lambda\)}
    \put(30,31){\(P_1\)}
    \put(63,31){\(P_m\)}
    \put(46,31){\(\cdots\)}
\end{overpic}
\caption{A subdivision with one unrooted cell with left and right handles.}
\label{fig:handle}
\end{figure}

\begin{proposition}
\label{prop:one-unrooted-cell-subdivision}
Each pair \((P,Q')\) satisfying these conditions determines a unique
coarse regular geometric subdivision into the unrooted cell
\(Q'\) and rooted cells. Conversely, every geometric coarse
subdivision with exactly one unrooted cell has this form.
\end{proposition}

\begin{proof}
Work in a developing coordinate.
For each positive edge \(\eta_\nu\), let \(f_\nu\) be the affine
function vanishing on its supporting line and satisfying
\(f_\nu(r)=1\). On \(P\), the convex function
\[
f=\max(0,f_1,\ldots,f_m)
\]
has \(Q'\) as its zero region and the required rooted regions as
its other maximal linearity regions.
Assigning the height \(f(x)\) to every marked point \(x\in C\)
gives the geometric cell markings, and proves regularity.

Any piecewise-affine function linear on these cells is determined
by its restriction to \(Q'\) and its value at \(r\).
Modulo global affine functions, the strictly convex choices
therefore form a one-dimensional cone. The subdivision is coarse.

Conversely, in a subdivision with only one unrooted cell,
each edge of that cell visible from \(r\) borders a rooted cell.
Its negative edges must lie on the outer boundary: a rooted cell
cannot lie across such an edge. The remaining boundary portions
are the two handles. Any extra radial cuts refine the subdivision
defined by \(f\). Since that subdivision is already coarse, a
coarse subdivision with this unrooted cell must coincide with it.
\end{proof}

Let \(N_\lambda,N_\rho\) be the linear tensor products along the
handles, and let \(L_-\) be the linear tensor product along
\(\partial^-Q'\), directed from \(a\) to \(b\).
An empty handle contributes the identity bimodule at its endpoint.

Write \(B_\nu=M^+\cap P_\nu\).
For \(m\geq2\), the unshifted rooted coefficient complexes are
\[
L_{B_1}=N_\lambda\otimes_{S_a}N_{\eta_1},
\qquad
L_{B_\nu}=N_{\eta_\nu}\quad(2\leq\nu\leq m-1),
\qquad
L_{B_m}=N_{\eta_m}\otimes_{S_b}N_\rho.
\]
If \(m=1\), both handles belong to the same rooted cell:
\[
L_{B_1}
=
N_\lambda\otimes_{S_a}N_{\eta_1}\otimes_{S_b}N_\rho.
\]
In either case,
\[
L_C=N_\lambda\otimes_{S_a}L_-\otimes_{S_b}N_\rho.
\]
These are statements about the unshifted boundary coefficient
complexes. The complete summands \(F_B^{\mathcal N}\) also contain
their own orientation lines and shifts.

Contracting the oppositely oriented copies of the positive edges
gives the basic coefficient map
\[
\gamma_{Q'}^{\mathrm{coeff}}:
N_{A'}\otimes_\Bbbk
N_{\eta_1}\otimes_\Bbbk\cdots\otimes_\Bbbk N_{\eta_m}
\longrightarrow L_-.
\]
The contraction for \((P,Q')\) is obtained by attaching the two
handle factors:
\[
\Gamma_{P,Q'}^{\mathrm{coeff}}
=
\operatorname{Id}_{N_\lambda}
\otimes
\gamma_{Q'}^{\mathrm{coeff}}
\otimes
\operatorname{Id}_{N_\rho},
\]
using the balanced tensor products at the endpoints and the
necessary graded permutations.
Tensoring this map with the cellular incidence map on orientation
lines, and applying the suspension convention, gives
\[
\Phi^C_{A';B_1,\ldots,B_m}:
E_{A'}^{\mathcal N}\otimes
F_{B_1}^{\mathcal N}\otimes\cdots\otimes F_{B_m}^{\mathcal N}
\longrightarrow F_C^{\mathcal N},
\qquad
\deg\Phi^C_{A';B_1,\ldots,B_m}=1-m.
\]

\begin{proposition}
\label{prop:matrix-elements-Phi}
The maps just described are precisely the components of the first
Taylor map \(\Phi_1\) indexed by coarse subdivisions with one
unrooted cell. Every component of \(\Phi_1\) is obtained this way.
\end{proposition}

\begin{proof}
By construction, a component of \(\Phi_1\) is a coefficient
contraction and an incidence map for a coarse subdivision with one
unrooted cell. Proposition~\ref{prop:one-unrooted-cell-subdivision} classifies
these subdivisions. The contractions along the handles are the
identity maps, while the other contractions are exactly
\(\gamma_{Q'}^{\mathrm{coeff}}\).
\end{proof}

\subsection{An explicit winding example}
\label{subsec:winding-example}

We exhibit two relative data with the same marked points on the
elliptic curve, the same root, and the same angular order, whose
geometric rooted algebras are not \(A_\infty\)-quasi-isomorphic.
Let
\[
E=\mathbb C/(\mathbb Z+i\mathbb Z),\qquad \eta=dz,
\qquad p=[0],
\]
and take the three marked points
\[
a=[1/4],\qquad b=[(1+i)/12],\qquad c=[i/4].
\]
Choose relative path classes \(\kappa_j^{(\nu)}:p\to j\),
where \(j\in\{a,b,c\}\) and \(\nu\in\{0,1\}\), with periods
\begin{equation}\label{eq:winding-example-periods}
\begin{aligned}
(q_a^{(0)},q_b^{(0)},q_c^{(0)})
 &=\left(\frac14,\frac{1+i}{12},\frac{i}{4}\right),\\
(q_a^{(1)},q_b^{(1)},q_c^{(1)})
 &=\left(\frac14,\frac{13+i}{12},\frac{i}{4}\right),
\qquad q_j^{(\nu)}=\int_{\kappa_j^{(\nu)}}\eta.
\end{aligned}
\end{equation}
Thus \(\kappa_b^{(1)}=\gamma*\kappa_b^{(0)}\), where
\(\gamma\) is a loop at \(p\) with period \(1\), and the
other two classes are unchanged. The induced classes between
non-root marks are \((\kappa_j^{(\nu)})^{-1}*\kappa_k^{(\nu)}\).
In particular, the classes involving \(b\) change as well; we are
not fixing the entire unrooted homotopy datum.

To describe the affine domains, normalize the development of the
root to \(r=0\), and write \(a_\nu=q_a^{(\nu)}\),
\(b_\nu=q_b^{(\nu)}\), and \(c_\nu=q_c^{(\nu)}\).
Let \(P_\nu^+=\operatorname{Conv}(r,a_\nu,b_\nu,c_\nu)\),
with these four selected marks and its natural map to \(E\).
These define relative data \(\mathfrak d_\nu\).
Both angular orders are \(a<b<c\), and both roots are exterior
to the corresponding unrooted triangle. However,
\[
b_0=\frac{r+a_0+c_0}{3}
\]
lies strictly inside the triangle \(ra_0c_0\), whereas
\(r,a_1,b_1,c_1\) are the successive vertices of a convex
quadrilateral. No three of the four marks are collinear in either
case. The factorization property is automatic here: for every
proper regular subdivision of a four-point configuration in general
position, each cell has three marks, so the target of
\eqref{eq:factorization-restriction} is zero; for the trivial
subdivision the restriction map is the identity.

Use the trivial coefficient system: every vertex algebra and edge
bimodule is \(\Bbbk\) in degree zero, and every contraction is
multiplication. Write
\[
R^{(\nu)}=R^+_{\mathcal N,\mathfrak d_\nu}.
\]
For these coefficients, a geometric rooted marking \(B\) contributes
one generator of degree \(|B|-2\).

\begin{proposition}\label{prop:winding-non-quasi-isomorphic}
The algebras \(R^{(0)}\) and \(R^{(1)}\) have zero differential
and graded cohomology dimensions
\begin{equation}\label{eq:winding-example-cohomology}
\begin{aligned}
\sum_k\dim_\Bbbk H^k(R^{(0)})T^k&=3+2T+T^2,\\
\sum_k\dim_\Bbbk H^k(R^{(1)})T^k&=3+3T+T^2.
\end{aligned}
\end{equation}
In particular, changing a single path class by a period can change
the \(A_\infty\)-quasi-isomorphism class of the geometric rooted
algebra, even with the marked base points, root, and angular order
fixed.
\end{proposition}

\begin{proof}
Besides the three degree-zero identity elements, the geometric
rooted generators are as follows. A string of vertex labels denotes
the indicated marking, not just the vertices of its convex hull.
\[
\begin{array}{c|c|c|c|c}
\text{generator}&\text{marking}&\text{summand}&\text{degree}
 &\text{present for}\\ \hline
x_\nu&\{r,a_\nu,b_\nu\}&R^{(\nu)}_{ab}&1&\nu=0,1\\
y_\nu&\{r,b_\nu,c_\nu\}&R^{(\nu)}_{bc}&1&\nu=0,1\\
q_\nu&\{r,a_\nu,b_\nu,c_\nu\}&R^{(\nu)}_{ac}&2&\nu=0,1\\
t_1&\{r,a_1,c_1\}&R^{(1)}_{ac}&1&\nu=1
\end{array}
\]
For \(\nu=0\), the marking \(\{r,a_0,c_0\}\) is not geometric,
because its triangle contains \(b_0\). The same triangular support
with its full geometric marking contributes \(q_0\), in degree
two. For \(\nu=1\), the triangle \(ra_1c_1\) does not contain
\(b_1\), and therefore contributes the additional generator \(t_1\).
This exhausts all rooted subconfigurations with two-dimensional
convex hull.

The coefficient differential vanishes, so
Proposition~\ref{prop:rooted-triangular} gives \(m_1=0\) and
\(m_k=0\) for \(k\geq3\). The only possible product of two
nonidentity generators is \(m_2(x_\nu,y_\nu)\).
For \(\nu=0\), the two input triangles have nonconvex union:
they leave out the triangle \(a_0b_0c_0\) from their convex hull.
Their product is therefore zero. For \(\nu=1\), they triangulate
the quadrilateral along \([r,b_1]\). Choosing the orientation of
\(q_1\) suitably gives
\begin{equation}\label{eq:winding-example-products}
m_2(x_0,y_0)=0,\qquad m_2(x_1,y_1)=q_1.
\end{equation}
Counting the generators proves \eqref{eq:winding-example-cohomology}.
An \(A_\infty\)-quasi-isomorphism would induce an isomorphism of
graded cohomology, which is impossible.
\end{proof}

\subsubsection*{The mixed action and its deformations.}
The unrooted algebra \(\mathfrak g^{(\nu)}\) has a single
degree-one generator \(u_\nu\), corresponding to the triangle
\(a_\nu b_\nu c_\nu\), and all its brackets vanish.
The mixed action is also explicit. After choosing generators of the
orientation lines, its only nonzero components on nonidentity
rooted inputs are
\begin{equation}\label{eq:winding-example-action}
\begin{aligned}
\bigl(\Phi^{(0)}_1(u_0)\bigr)_2(x_0,y_0)&=q_0,\\
\bigl(\Phi^{(1)}_1(u_1)\bigr)_1(t_1)&=q_1.
\end{aligned}
\end{equation}
Here the subscript outside parentheses denotes the rooted arity.
The first formula comes from the star subdivision of
\(ra_0c_0\) into \(ra_0b_0\), \(rb_0c_0\), and
\(a_0b_0c_0\). This is a coarse regular subdivision: its face
is an endpoint of the one-dimensional secondary polytope of the
four-point configuration. The other endpoint omits \(b_0\) and
does not have geometric input marking. The second formula comes
from the triangulation of \(ra_1b_1c_1\) along \([a_1,c_1]\),
with rooted cell \(ra_1c_1\) and unrooted cell \(a_1b_1c_1\).
The other triangulation gives the binary product in
\eqref{eq:winding-example-products}. Since there is only one
unrooted triangle, no subdivision has two unrooted cells, and
\(\Phi^{(\nu)}_n=0\) for \(n\geq2\) in this example.

For any scalar \(\epsilon\in\Bbbk\), the element
\(\epsilon u_\nu\) is Maurer--Cartan. Its action gives the
following deformed operations, in addition to the identity actions:
\begin{equation}\label{eq:winding-example-deformations}
\begin{array}{c|c|c}
 &\text{unary operation}&\text{binary operation}\\ \hline
R^{(0),\epsilon u_0}&m_1=0&m_2(x_0,y_0)=\epsilon q_0\\
R^{(1),\epsilon u_1}&m_1(t_1)=\epsilon q_1&m_2(x_1,y_1)=q_1
\end{array}
\end{equation}
All unlisted operations on nonidentity inputs vanish. The deformation
is a finite sum, so no completion is needed. For \(\epsilon\ne0\),
the pair \(t_1,q_1\) is acyclic, and consequently
\[
\begin{aligned}
\sum_k\dim_\Bbbk H^k(R^{(0),\epsilon u_0})T^k&=3+2T+T^2,\\
\sum_k\dim_\Bbbk H^k(R^{(1),\epsilon u_1})T^k&=3+2T.
\end{aligned}
\]
Thus these algebraically chosen deformations are not
\(A_\infty\)-quasi-isomorphic either.

\section{Elliptic universality and directed deformations}\label{uni-thm}

We state the relation between the unrooted and rooted algebras
 over the elliptic curve. The mixed subdivisions give
an \(L_\infty\)-morphism to a directed deformation complex.
For scalar vertex algebras, its restriction to the strictly
positive-area part is a quasi-isomorphism. The comparison with
the planar construction is given in
Remark~\ref{rem:elliptic-planar-universality}.

Fix a finite intrinsic relative datum
\(\mathfrak d=(\mathcal P^+,r)\) over \((E,\eta)\), satisfying
general position and the factorization property of
Definition~\ref{def:factorization-genericity}. For an extended coefficient
system \(\mathcal N\) as in Definition~\ref{def:extended-coeff},
write
\[
\mathfrak g=\mathfrak g_{\mathcal N},\qquad
R^+=R^+_{\mathcal N,\mathfrak d}
=\bigoplus_{i\le j}R_{ij},\qquad R_{ii}=S_i.
\]
The indices follow the angular order at the root.
The undeformed algebra \(R^+\) is a dg algebra, with coefficient
differential \(m_1\), convex-gluing product \(m_2\), and
\(m_k=0\) for \(k\ge3\). The diagonal dg algebras and their
actions on the off-diagonal bimodules are fixed.

\begin{definition}\label{def:directed-hoch-chamber}
For \(k\ge1\), put
\[
\begin{split}
\overrightarrow C^{\,k}(R^+,R^+)
={}&\bigoplus_{i_0<\cdots<i_k}
\operatorname{Hom}_{S_{i_0}\otimes_\Bbbk S_{i_k}^{\mathrm{op}}}
\Bigl(
R_{i_0i_1}\otimes_{S_{i_1}}\cdots
\otimes_{S_{i_{k-1}}}R_{i_{k-1}i_k},
R_{i_0i_k}
\Bigr),\\
\mathfrak h_{\mathrm{dir}}
:={}&\overrightarrow C^{\,\ge1}(R^+,R^+)[1]
=\bigoplus_{k\ge1}
\overrightarrow C^{\,k}(R^+,R^+)[1-k].
\end{split}
\]
Here \(\operatorname{Hom}\) is the internal Hom complex.
An arity-\(k\) cochain of internal degree \(q\) has total
degree \(q+k-1\).
\end{definition}

Using the suspended coderivation convention, the differential
is \(d_{\mathrm H}=[b,-]\), where \(b\) is the bar
codifferential of \(R^+\), and the bracket is the graded
commutator. Directed cochains extend by zero on diagonal inputs.
Balancing and endpoint linearity make them a dg Lie subalgebra
of the normalized Hochschild complex. There are no arity-zero
cochains, and the arity sum is finite.

\begin{definition}[Area excess]\label{def:positive-area-target}
For a polygonal domain \(P\), let
\[
\mathcal A(P)=\int_P\rho^*
\left(\frac{\mathrm i}{2}\eta\wedge\overline\eta\right).
\]
Thus area is measured on the domain, including multiplicity
when its image in \(E\) overlaps itself. Write \(F_P\) for
the rooted summand with support \(P\) and its geometric
marking. A directed matrix-cochain summand
\[
\operatorname{Hom}(F_{P_1}\otimes\cdots\otimes F_{P_k},F_{P_0})[1-k]
\]
has \emph{area excess}
\[
w(P_0;P_1,\ldots,P_k)
=\mathcal A(P_0)-\sum_{a=1}^k\mathcal A(P_a).
\]
The Hom and tensor products have the balancing conventions
above. Let
\[
\mathfrak h_{>0}(\mathfrak d,\mathcal N)
=\overrightarrow C^{\,\ge1}_{>0}(R^+,R^+)[1]
\]
be the direct sum of the matrix-cochain summands of strictly
positive area excess.
\end{definition}

The coefficient differential preserves each polygon label.
The product \(m_2\) glues polygons with disjoint interiors,
so their areas add. Hence \(d_{\mathrm H}\) preserves area
excess. In a cochain insertion the intermediate polygon
occurs once as an output and once as an input; its area
cancels. Thus excesses add under the bracket, and
\(\mathfrak h_{>0}\) is a dg Lie subalgebra.

For the quasi-isomorphism statement, assume from now on that
\(S_i=\Bbbk\) for every \(i\). Each oriented edge carries a
bounded finite-dimensional complex \(N_\sigma\), with the
perfect degree-zero chain pairings
\(N_\sigma\otimes N_{\bar\sigma}\to\Bbbk\) of
Section~\ref{coeffecient-system}. These hypotheses include the
trivial coefficient system.

\begin{theorem}[Elliptic universality]
\label{universality-elliptic}
\label{thm:polygonal-action-deformation}
Under the hypotheses above:
\begin{enumerate}
\item
The mixed subdivisions define an \(L_\infty\)-morphism
\[
\Psi_{\mathfrak d}:\mathfrak g_{\mathcal N}
\longrightarrow \mathfrak h_{>0}(\mathfrak d,\mathcal N).
\]
This morphism is a quasi-isomorphism: its first Taylor
component induces an isomorphism on cohomology.

\item
For every Maurer--Cartan element
\(\gamma\in\mathfrak g_{\mathcal N}^1\), the element
\begin{equation}\label{eq:polygonal-mc-pushforward}
\beta_\gamma
=\sum_{n\ge1}\frac1{n!}
\Psi_{\mathfrak d,n}(\gamma,\ldots,\gamma)
\end{equation}
is Maurer--Cartan in \(\mathfrak h_{>0}\). In particular,
\begin{equation}\label{eq:polygonal-deformed-operations}
m_k^\gamma
=m_k+\sum_{n\ge1}\frac1{n!}
\bigl(\Psi_{\mathfrak d,n}(\gamma,\ldots,\gamma)\bigr)_k,
\qquad k\ge1,
\end{equation}
defines a directed, uncurved, strictly unital
\(A_\infty\)-algebra \(R_{\mathcal N,\mathfrak d}^{+,\gamma}\).
Its underlying graded spaces and diagonal units are fixed.
\end{enumerate}
The sums are finite for the fixed finite polygonal datum.
\end{theorem}

\begin{proof}
Write \(C=\mathfrak h_{>0}(\mathfrak d,\mathcal N)\).
Proposition~\ref{prop:coefficient-polygonal-action} gives the
mixed-subdivision \(L_\infty\)-morphism. Its rooted inputs
occur in consecutive angular sectors, so its components are
directed. For a subdivision with rooted output \(P_0\),
rooted inputs \(P_1,\ldots,P_k\), and unrooted inputs
\(Q_1,\ldots,Q_n\), area additivity gives
\[
w(P_0;P_1,\ldots,P_k)
=\sum_{b=1}^n\mathcal A(Q_b)>0
\qquad(n\ge1).
\]
Thus the full morphism lands in \(C\). Its Taylor components
vanish for sufficiently large \(n\), since the fixed marked
configuration admits only finitely many subdivisions.
It remains to prove that \(\Psi_{\mathfrak d,1}\) is a
quasi-isomorphism. All polygonal arguments below take place
on the affine domain \(P^+\).

\smallskip
\noindent\textbf{Step 1. Cochains and boundary paths.}
A matrix-cochain summand is indexed by rooted polygons
\(P_0,P_1,\ldots,P_k\), where \(P_a\) has root edges ending
at \(m_{i_{a-1}},m_{i_a}\) and \(i_0<\cdots<i_k\), while
\(P_0\) has root edges ending at \(m_{i_0},m_{i_k}\).
Let \(\alpha\) be the concatenation of the non-root
boundaries of the input polygons, and let \(\beta\) be the
non-root boundary of \(P_0\), both oriented in increasing
angular order. They have the same endpoints and determine
the closed edge path \(\alpha*\beta^{-1}\).
We do not assume that the input polygons lie inside \(P_0\).

A rooted polygon is the union of its fan triangles
\[
T_{ij}=\operatorname{Conv}_{P^+}(r,m_i,m_j)
\]
along consecutive edges of its non-root boundary. Each
triangle carries its full geometric marking; in particular,
its summand need not be concentrated in degree one.
Since \(S_i=\Bbbk\) and the root edges carry no additional
coefficient factors, the iterated rooted product gives an
isomorphism from the tensor product of the fan summands to
the polygon summand. These isomorphisms include the
orientation lines and shifts and are compatible with
grouping the triangles, by associativity.

For the convexity argument, develop the domain with \(r=0\)
and choose real coordinates \((s,t)\) with \(s>0\) at every
non-root mark. In the projective chart
\[
(u,v)=(t/s,1/s),
\]
a rooted polygon is the epigraph of a convex polygonal
graph, and angular order is increasing \(u\).
This chart is used only for convexity; areas are always
computed in the original affine structure.
Consequently, a boundary chain defines a rooted polygon
exactly when all its interior turns are convex.
General position excludes flat turns.

Thus a cochain label consists of the boundary chains
\(\alpha,\beta\), with \(\beta\) convex, together with a
partition of \(\alpha\) into convex blocks.
For fixed \(\alpha,\beta\), the fan-product isomorphisms
identify the input tensor products for all such partitions.
The corresponding Hom complexes differ only by their
Hochschild arity shifts. The sum of the input areas is
also independent of the partition.

\smallskip
\noindent\textbf{Step 2. The handle filtration.}
Let \(\lambda\) and \(\rho\) be the maximal common initial
and final edge chains of \(\alpha\) and \(\beta\).
They are the two handles of the closed boundary path.
Because \(w>0\), the chains \(\alpha,\beta\) cannot coincide.
Their common initial and final chains therefore have
disjoint interiors. Put
\[
h(\alpha,\beta)=\#\operatorname{Edges}(\lambda)
+\#\operatorname{Edges}(\rho),
\]
and let \(G^pC\) be the direct sum of cochain summands with
\(h(\alpha,\beta)\ge p\). This is a finite decreasing
filtration, with \(G^0C=C\).

Write the Hochschild differential as the coefficient
differential plus the terms induced by \(m_2\).
The middle terms split an input polygon into two rooted
polygons; they preserve \(\alpha,\beta\), and hence \(h\).
Each nonzero outer term adjoins the same nonempty boundary
chain to \(\alpha\) and \(\beta\) on the left or on the
right, strictly increasing \(h\).
The coefficient differential preserves each polygon label.
Thus \(G^\bullet C\) is a filtration by subcomplexes.
On \(\operatorname{gr}_G C\), only the coefficient
differential and the input-splitting terms remain.

\smallskip
\noindent\textbf{Step 3. The part without handles.}
For fixed \(\alpha,\beta\), every concave turn of \(\alpha\)
must be a cut of its partition into convex blocks, and
every convex turn is an optional cut. The corresponding
associated-graded complex runs over all subsets of these
optional cuts.

If an optional cut exists, this complex is contractible.
Indeed, choose one such cut and write its graded space as
\(H_0\oplus H_1\), according to whether the cut is absent
or present. The differential has the form
\[
D(x,y)=(D_0x,\ a x+D_1y),
\]
where \(a:H_0\to H_1\) is the degree-one map adding the
chosen cut, including its actual Hochschild sign.
It is invertible as a graded map: its matrix components
are precomposition with the fan-product isomorphisms.
The identity \(D^2=0\) gives \(D_1a+aD_0=0\).
Hence
\[
H(x,y)=(a^{-1}y,0)
\qquad\text{satisfies}\qquad
DH+HD=\operatorname{id}.
\]
This proves contractibility with the coefficient
differential and all suspension signs included.

Suppose now that there are no optional cuts.
Every interior turn of \(\alpha\) is concave, and every
input polygon is a fan triangle. In the projective chart,
\(\alpha\) lies above the chord joining its endpoints and
\(\beta\) lies below it. Unless they are the same single
edge, these chains bound a nondegenerate convex unrooted
polygon \(Q\). The single-edge case has \(w=0\) and is
absent from \(C\). In the other case, \(\alpha\) is the
root-visible boundary of \(Q\), \(\beta\) is its
complementary boundary, and
\[
w=\mathcal A(Q)>0.
\]
Such a pair has no common initial or final edge:
at an interior point of a common edge both graphs would
have to lie on their common endpoint chord, contrary to
strict convexity or concavity of any nontrivial chain.
Thus every surviving pair lies in \(\operatorname{gr}_G^0C\).

Write \(e_1,\ldots,e_k\) for the visible edges of \(Q\)
and \(\widehat Q=\operatorname{Conv}_{P^+}(r,Q)\).
Its surviving complex is
\[
K_Q=
\operatorname{Hom}_{\Bbbk}
\left(
F_{T_{e_1}}\otimes\cdots\otimes F_{T_{e_k}},
F_{\widehat Q}
\right)[1-k],
\]
with its internal coefficient differential.
The perfect edge pairings identify the duals of the
visible-edge coefficients with the reverse-edge
coefficients. Together with the complementary boundary
factors, these give precisely the cyclic complex \(N_{B_Q}\).

For completeness, let \(B_0,B_Q,B_1,\ldots,B_k\) be the
full markings of \(\widehat Q,Q,T_{e_1},\ldots,T_{e_k}\).
Counting the repeated vertices in this subdivision gives
\[
|B_0|=|B_Q|+\sum_{a=1}^k|B_a|-3k+1.
\]
There are no further marks on its shared edges, by general
position. Therefore the polygonal shifts satisfy
\[
(|B_0|-2)-\sum_{a=1}^k(|B_a|-2)+k-1=|B_Q|-2,
\]
as required for the source summand \(E_{B_Q}^{\mathcal N}\).

Let \(q:C\to C/G^1C=\operatorname{gr}_G^0C\) be the
quotient map. By
Propositions~\ref{prop:one-unrooted-cell-subdivision}
and~\ref{prop:matrix-elements-Phi}, the component of
\(q\Psi_{\mathfrak d,1}\) associated to \(Q\) is the
subdivision of \(\widehat Q\) into \(Q\) and its visible
fan triangles. Transposing its coefficient contraction
gives the perfect-duality isomorphism just described.
Its orientation factor is the nonzero incidence
isomorphism of that coarse subdivision, with the
suspension already used to define \(\Psi_{\mathfrak d,1}\).
We introduce no independent choice of signs.

Every other component of \(\Psi_{\mathfrak d,1}\) has a
handle and vanishes after applying \(q\).
Consequently \(q\Psi_{\mathfrak d,1}\) identifies the
geometric source with \(\bigoplus_Q K_Q\) as a graded
space. It is already a chain map, since
\(\Psi_{\mathfrak d,1}\) and \(q\) are chain maps, so this
is an identification of complexes.
All remaining summands of \(\operatorname{gr}_G^0C\)
are the contractible splitting complexes above.
Thus \(q\Psi_{\mathfrak d,1}\) is a quasi-isomorphism.

\smallskip
\noindent\textbf{Step 4. The pieces with handles.}
Consider \(\operatorname{gr}_G^pC\) with \(p\ge1\).
For fixed \(\alpha,\beta\), the absence of optional cuts
would give one of the pairs classified in Step 3:
either an identical single edge, excluded by \(w>0\),
or the two boundary arcs of a convex polygon, which have
no handles. Both contradict \(p\ge1\).
Hence every such summand has an optional cut and is
contractible by the displayed homotopy.
It follows that
\[
H^\bullet(\operatorname{gr}_G^pC)=0
\qquad(p\ge1).
\]
The filtration is finite, so \(G^1C\) is acyclic and
\(q:C\to C/G^1C\) is a quasi-isomorphism.
Since \(q\Psi_{\mathfrak d,1}\) is also a quasi-isomorphism,
so is \(\Psi_{\mathfrak d,1}\). This proves (1).

Finally, the \(L_\infty\)-morphism identities applied to a
Maurer--Cartan element \(\gamma\) give
\[
d_{\mathrm H}\beta_\gamma
+\tfrac12[\beta_\gamma,\beta_\gamma]=0.
\]
Equivalently, \((b+\beta_\gamma)^2=0\).
Its Taylor components are the operations in
\eqref{eq:polygonal-deformed-operations}.
They have positive arity, strictly increasing indices,
and vanish on diagonal inputs. The resulting structure
is therefore uncurved, directed, and strictly unital,
with the stated fixed underlying data.
The sums are finite by the subdivision bound above.
This proves (2).
\end{proof}

\begin{remark}[Relation with the planar universality construction]
\label{rem:elliptic-planar-universality}
Affine development identifies the intrinsic polygonal source,
rooted algebra, and mixed-subdivision morphism with the
corresponding planar constructions of
\cite[Sections~10--12]{KKS}. It also preserves area excess.
Thus the theorem gives an elliptic formulation of universality
for the positive-area deformation complex, under the scalar
coefficient hypotheses stated above.

The target here is the positive-area subcomplex of the
directed complex appearing in \cite[Theorem~12.2]{KKS}.
The quasi-isomorphism is proved above for this target.
Development is a strict identification of all the polygonal
operations; it also identifies the deformations associated
to corresponding Maurer--Cartan elements.
\end{remark}

\begin{remark}[Why the positive-area target is specified]
\label{rem:universality-normalization}
The same proof applied to the full directed complex retains,
in addition to the polygon summands, the zero-excess
endomorphism complexes \(\operatorname{End}_\Bbbk(F_{T_{ij}})\).
Thus the area restriction removes an entire family of
complexes, not merely a cohomological degree.

Already for trivial coefficients, the full target has
degree-zero cocycles
\[
D_c|_{R_{ij}}=(c_j-c_i)\operatorname{id}_{R_{ij}},
\qquad D_c|_{R_{ii}}=0.
\]
They are derivations because the differences add along a
composable pair. Since the coefficient differential vanishes
and the Hochschild differential raises arity, a nonzero
unary cocycle of this form cannot be a boundary in the
positive-arity complex. The polygonal source, with trivial
coefficients, is concentrated in positive degrees.
The quasi-isomorphism assertion therefore concerns
\(\mathfrak h_{>0}\), rather than the unrestricted target.
\end{remark}

\section{Fukaya--Seidel categories of curve-valued Lefschetz fibrations}
\label{sec:mc-fs-torus}

We explain how the preceding polygonal algebras are expected to
model a directed Fukaya--Seidel category for a potential
\(f:X\to E\). The construction has three parts: a geometric
coefficient system \(\mathcal N\), an instanton-counting
Maurer--Cartan element in \(\mathfrak g_{\mathcal N}\), and a
comparison of the resulting deformation of the rooted algebra
with a directed Floer algebra. All admissible paths lie on the
base curve, and all thimbles, solitons, and instantons lie in
\(X\).

For the Fukaya--Seidel construction we refer to
\cite{Seidel,SeidelVanishingMutation}; for the curve-base
framework used here, see \cite[Section~5]{KS}. The proposed
instanton comparison follows the perspective of
\cite{GMW15,KKS,HaydysFS}. Analytic results for a complex
gradient-flow construction under suitable hypotheses are given
in \cite{WangComplexGradient}; the distinction between that
construction and the polygonal instanton proposal is discussed
in \cite[Section~2.2]{DoanRezchikovFueter}.

\subsection{Fukaya--Seidel category on a base curve}
\label{sec:fs-surface}

Let \((X,\omega=d\lambda,I)\) be an exact K\"ahler manifold,
let \(Y\) be a complex curve, and let
\[
f:X\longrightarrow Y
\]
be holomorphic with finitely many Morse critical points
\(x_1,\ldots,x_N\) and distinct critical values
\(A=\{w_1,\ldots,w_N\}\). Fix a regular value \(z\) and a
distinguished oriented tangent ray at \(z\). We also fix grading
and relative Pin data for the branes. A suitable trivialization
of the squared canonical bundle, or a holomorphic volume form,
provides the grading structure.

We work in a Floer-theoretic framework in which parallel
transport along the chosen paths is defined, the chosen
thimbles are admissible exact branes, and coherent perturbation
data give compactified moduli spaces with the boundary
identities required for an \(A_\infty\)-category. In the
noncompact setting these are additional hypotheses on the
geometry at infinity; exactness alone does not provide them.

An admissible path starts at \(z\), ends at one critical value,
and otherwise avoids \(A\). Its initial tangent ray avoids the
distinguished ray. We use embedded paths when working with
embedded thimbles. More general immersed paths are included
only when their thimbles and the necessary brane data are
admissible in the chosen Floer model.

\begin{figure}[h]
\centering
\begin{overpic}[width=0.6\textwidth]{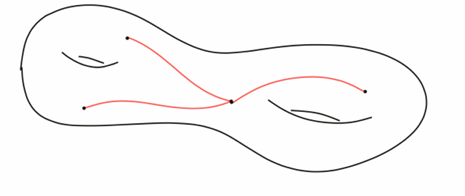}
    \put(48,23){$z$}
\end{overpic}
\caption{Admissible paths on the base curve.}
\label{fig:FS}
\end{figure}

Parallel transport along \(\delta\) defines a vanishing cycle
\(V_\delta\subset F_z=f^{-1}(z)\) and a thimble
\(T_\delta\subset X\). These are different objects: the
vanishing cycle lies in the fiber, while the thimble lies over
the entire path. Put
\[
X_z=X\setminus F_z,
\qquad L_\delta=T_\delta\setminus F_z,
\]
with the chosen brane structures. Near the removed fiber, use
the distinguished tangent ray to specify the one-sided Floer
perturbations. These perturbations are data on \(X_z\); a
diffeomorphism of the base alone is not a perturbation of a
Lagrangian in \(X_z\).

Write \(CF^+_z(L_0,L_1)\) for the resulting morphism complex.
Together with coherent polygon counts these complexes define
the thimble category \(A_z(f)\) in the chosen framework. We use
the convention
\[
\operatorname{FS}_z(f)=\operatorname{Tw}(A_z(f)),
\]
and take its idempotent completion when a split-closed category
is intended.

Now fix a finite ordered collection
\(S=(\delta_1,\ldots,\delta_N)\), with
\(\delta_i(1)=w_i\), and compatible directed Floer data.
Its directed category \(A_S\) has morphisms
\[
\operatorname{Hom}_{A_S}(i,j)=
\begin{cases}
CF^+_z(L_{\delta_i},L_{\delta_j}),&i<j,\\
\Bbbk e_i,&i=j,\\
0,&i>j.
\end{cases}
\]
For strictly increasing sequences the operations are the
corresponding Floer operations, and identities are adjoined
strictly. This defines the directed total algebra
\[
R_S=\bigoplus_{i\leq j}\operatorname{Hom}_{A_S}(i,j).
\]
It is a directed presentation, and is a full subcategory of
\(A_z(f)\) only when the chosen geometric model supplies that
identification. Merely setting reverse morphisms to zero does
not establish fullness. Likewise, identifying
\(\operatorname{Tw}(A_S)\) with the entire
\(\operatorname{FS}_z(f)\) requires a generation theorem.

Path dependence is recorded in the punctured base
\(Y^\circ=Y\setminus A\), with terminal germs at the critical
values. In particular, loops used to modify paths belong to
\(\pi_1(Y^\circ,z)\). Their symplectic parallel transport acts
on the fiber and its branes. The unpunctured group
\(\pi_1(Y,z)\) forgets the loops around critical values and
therefore does not contain all this information.

\subsection{The elliptic complex Morse model and Maurer--Cartan element}
\label{subsec:lifted-complex-morse-model}

Specialize to an elliptic base \(E\) with a fixed nonzero
holomorphic one-form \(\eta\). The global datum replacing the
differential of a complex-valued potential is
\[
\alpha=f^*\eta\in\Omega^1_{\mathrm{hol}}(X).
\]
It is closed and vanishes precisely at the critical points of
\(f\). If \(w_U\) is a local primitive of \(\eta\) on a base
chart \(U\), then \(W_U=w_U\circ f\) satisfies
\(dW_U=\alpha\). On overlaps the functions \(W_U\) differ by
constants. Thus their gradient vector fields agree, even when
there is no global complex-valued function on \(X\) with
differential \(\alpha\).

\begin{remark}[Relation to the physics of twisted masses]
\label{rem:fs-khan-moore}
The replacement of \(dW\) by a closed holomorphic one-form
relates this proposal to Khan--Moore's extension of the algebra
of the infrared to theories with twisted masses
\cite[Section~3]{KhanMoore2024}. In the corresponding
Landau--Ginzburg description, the bosonic potential is
proportional to \(\lVert\alpha\rVert_g^2\), and the massive
vacua correspond to the Morse zeros of \(\alpha\).
With a fixed normalization, the
central charge of a path sector \(c\) joining critical points
is
\[
Z(c)=\int_c\alpha=\int_{f\circ c}\eta.
\]
Consequently, changing the relative path class changes the
central charge by a period. Nonzero periods can also allow
BPS states in sectors with the same vacuum at both ends;
their existence is an additional dynamical question.
This gives a physical interpretation of the winding data
retained in our elliptic formulation. Khan--Moore's detailed
web-category construction concentrates on a single vacuum
and a single independent period. It motivates, but does not
establish, the general elliptic comparison proposed here.
Our finite polygonal model does not automatically include
all winding sectors or the additional same-vacuum BPS sectors.
\end{remark}

Choose a relative affine polygonal datum
\[
\mathfrak d=(\mathcal P^+,r),\qquad
\mathcal P^+=(P^+,M^+,\rho),\qquad
\rho(r)=z,\quad \rho(m_i)=w_i,
\]
as in Section~\ref{sec:ainf-elliptic-notation}. Compatibility
with \(S\) includes the unpunctured path classes
\([\rho([r,m_i])]=[\delta_i]\), as well as the additional
punctured-path and transport data discussed below.
The selected affine edges are required to avoid every critical
value in their interiors after projection to \(E\).

For an oriented edge \(\sigma\) from \(m_i\) to \(m_j\), set
\[
Z_\sigma=\int_{\rho\circ\sigma}\eta,
\qquad
\zeta_\sigma=-\frac{\overline{Z_\sigma}}{|Z_\sigma|}.
\]
Then \(\zeta_\sigma Z_\sigma=-|Z_\sigma|\). For the
K\"ahler metric \(g\), define
\[
V_\zeta=g^\sharp\operatorname{Re}(\zeta\alpha).
\]
Locally this is \(\nabla\operatorname{Re}(\zeta W_U)\).

\begin{definition}
A soliton in the sector \(\sigma\) is a trajectory, modulo
translation of its parameter,
\[
\dot\phi(t)=-V_{\zeta_\sigma}(\phi(t)),
\qquad
\lim_{t\to-\infty}\phi(t)=x_i,\quad
\lim_{t\to+\infty}\phi(t)=x_j,
\]
whose projected path belongs to the prescribed sector
\(\rho\circ\sigma\), including its relative path data.
\end{definition}

Along such a trajectory the imaginary part of a local
\(\zeta_\sigma W_U\) is constant and its real part decreases.
Consequently its projection follows the corresponding straight
affine segment on \(E\), with its path label retained. Both
the equation and the trajectory are defined in \(X\).

The edge coefficient \(N_\sigma\) is the complex of soliton
states in this sector, with orientation lines and the grading
normalization required in Section~\ref{coeffecient-system}.
Under the usual transversality and compactness hypotheses it
can be modeled using the two vanishing cycles transported
along the two halves of \(\rho\circ\sigma\) to an interior
regular point of that path. Thus it is an edge-dependent
Floer complex; it is not, in general, the whole morphism
complex between the two thimbles in the fixed collection
\(S\).

For Morse critical values put \(S_i=\Bbbk\). Reversal and 
Floer duality must be expressed, with the appropriate grading
shifts and signs, as the degree-zero perfect pairings
\[
\beta_\sigma:N_\sigma\otimes N_{\bar\sigma}
\longrightarrow\Bbbk
\]
used earlier. We assume a compatible chain model satisfying
these coefficient hypotheses. Ordinary unshifted Floer
duality should not silently be identified with a degree-zero
pairing. The differential of \(N_\sigma\) is part of the
coefficient complex; the corresponding two-ended counts are
not counted a second time as unrooted polygon operations.

For a geometric unrooted marked polygon \((P_B,B)\), these
complexes give
\[
N_B=\bigotimes_{\sigma\subset\partial P_B}N_\sigma,
\qquad
E_B^{\mathcal N}=N_B\otimes\omega_B[-d_B-1],
\qquad d_B=|B|-3.
\]
The tensor factors follow the boundary orientation. The
resulting \(L_\infty\)-algebra is precisely
\(\mathfrak g_{\mathcal N}\) from the coefficient construction.

Fix a phase \(\zeta\in S^1\). A form of the Witten
equation, in one consistent real-coordinate convention, is
\begin{equation}
\label{eq:intrinsic-witten-fs}
\partial_tu+I(u)\partial_su+V_\zeta(u)=0,
\qquad u:\mathbb R^2\longrightarrow X.
\end{equation}
Its solutions independent of \(s\) satisfy the downward
gradient equation above. The local expressions obtained from
\(W_U\) agree because only \(dW_U=\alpha\) occurs.
Domain-dependent perturbations and end coordinates must be
chosen compatibly with this convention.

\subsubsection*{Cyclic soliton asymptotics.}
We describe the candidate moduli spaces explicitly, in the
spirit of \cite[Section~14.E]{KKS}. Let the boundary vertices
of \(P_B\) be \(m_{i_0},\ldots,m_{i_{k-1}}\) in
counterclockwise order, with indices understood modulo
\(k\), and put \(\sigma_a=[m_{i_a},m_{i_{a+1}}]\).
Choose an isolated unparametrized soliton \(\phi_a\) in
each sector \(\sigma_a\). The cyclic tuple
\[
\boldsymbol\phi=(\phi_0,\ldots,\phi_{k-1})
\]
will be called a soliton fan over \(P_B\). Every soliton
retains its prescribed path sector in \(E\). We describe
the isolated-soliton model; non-isolated soliton families
would require additional Morse--Bott data.

Write \(Z_a=Z_{\sigma_a}\) and
\(\zeta_a=\zeta_{\sigma_a}\), and choose angles
\(\beta_a\) by
\[
e^{i\beta_a}\zeta=\zeta_a.
\]
In the domain plane, whose complex coordinate is \(t+is\),
use the outgoing rays \(i e^{i\beta_a}\mathbb R_{\geq0}\)
and the transverse coordinates
\begin{equation}
\label{eq:fs-fan-coordinates}
\kappa_a(r,\tau)=e^{i\beta_a}(\tau+ir),
\qquad r\longrightarrow+\infty.
\end{equation}
These rays bound the sectors labelled by the critical
points \(x_{i_a}\): across the \(a\)-th ray, increasing
\(\tau\) passes from \(x_{i_a}\) to \(x_{i_{a+1}}\).
The convexity of \(P_B\) gives a consistent fan. The cyclic
order in the domain is determined by this equation and
need not be identified with the boundary orientation of
\(P_B\).

This choice of phases can be checked directly. If
\(v_a(r,\tau)=u(\kappa_a(r,\tau))\), then
\eqref{eq:intrinsic-witten-fs} becomes
\[
\partial_\tau v_a+I(v_a)\partial_r v_a+V_{\zeta_a}(v_a)=0.
\]
Here we use \(V_{e^{i\beta}\zeta}=e^{-\beta I}V_\zeta\),
which follows from the holomorphicity of \(\alpha\).
Consequently an \(r\)-independent end is exactly a soliton
in the required sector.

\begin{definition}[Candidate instanton moduli space]
\label{def:fs-instanton-moduli}
Fix a soliton fan \(\boldsymbol\phi\) and a relative
homotopy class \(\mathfrak c\) of maps with the following
asymptotics. Let
\(\widetilde{\mathcal M}_{\zeta,B}
(\boldsymbol\phi;\mathfrak c)\) consist of smooth solutions
\(u:\mathbb R^2\to X\) of
\eqref{eq:intrinsic-witten-fs} in that class such that:
\begin{enumerate}
\item In the sector labelled by \(x_{i_a}\), the map
\(u\) converges to \(x_{i_a}\) as the distance from the
bounding rays tends to infinity. The convergence is uniform
on closed angular subsectors and holds with derivatives.
\item For each ray there is a real number \(c_a\) such that
\[
u(\kappa_a(r,\tau))\longrightarrow\phi_a(\tau-c_a)
\quad\text{as }r\longrightarrow+\infty,
\]
in \(C^\infty_{\mathrm{loc}}\) in the transverse variable
\(\tau\). In an analytic realization, impose compatible
weighted decay conditions after subtracting these end
profiles. The shifts \(c_a\) are allowed to vary.
\item The renormalized energy defined below exists and is
finite.
\end{enumerate}
The reduced moduli space is
\begin{equation}
\label{eq:fs-reduced-moduli}
\mathcal M_{\zeta,B}(\boldsymbol\phi;\mathfrak c)
=\widetilde{\mathcal M}_{\zeta,B}
(\boldsymbol\phi;\mathfrak c)/\mathbb R^2,
\end{equation}
where \(\mathbb R^2\) acts by translation of the domain.
We also write
\(\mathcal M_{\zeta,B}(\boldsymbol\phi)
=\coprod_{\mathfrak c}
\mathcal M_{\zeta,B}(\boldsymbol\phi;\mathfrak c)\).
\end{definition}

The homotopy class is taken in \(X\) with the soliton
asymptotics retained, not merely in the base. Equivalently,
one records a filling class of the loop obtained by
concatenating the compactified soliton paths. If this loop
has no filling in \(X\), the moduli space is empty. No
condition requires \(f\circ u\) to parametrize \(P_B\),
or even to remain in \(\rho(P_B)\); any confinement needed
for the finite model must be proved separately.

Allowing the end shifts is necessary for the translation
action: fixing the origins of every soliton profile would
generally destroy that action. For a two-dimensional fan
with at least three distinct sector labels, translations
act freely. Neither dilations nor general holomorphic
reparametrizations are symmetries of the fixed equation.

\subsubsection*{Renormalized energy.}
Let \(D_R\subset\mathbb R^2\) be the Euclidean disc of
radius \(R\), and define
\[
\mathcal E_R(u)=\frac12\int_{D_R}
\left(|\partial_tu|_g^2+|\partial_su|_g^2
+|V_\zeta(u)|_g^2\right)\,dt\,ds.
\]
The energy per unit length of the \(a\)-th soliton end is
\[
T_a=\int_{\mathbb R}|\dot\phi_a|_g^2\,d\tau
=-\int_{\phi_a}\operatorname{Re}(\zeta_a\alpha)
=|Z_a|.
\]
Thus ordinary finite total energy would exclude these
ends. With this fixed exhaustion, the renormalized condition
in Definition~\ref{def:fs-instanton-moduli} is
\begin{equation}
\label{eq:fs-renormalized-energy}
\mathcal E^{\mathrm{ren}}(u)
=\lim_{R\to\infty}
\left(\mathcal E_R(u)-R\sum_{a=0}^{k-1}T_a\right)
\in\mathbb R.
\end{equation}
This is a condition on the chosen asymptotic class, not a
compactness theorem or an assertion of positivity. In a
perturbed model the subtraction must use its corresponding
end profiles.

\subsubsection*{Expected dimension and rigid counts.}
The unperturbed linearized operator, using the K\"ahler
connection, is
\[
D_u\xi=\nabla_t\xi+I\nabla_s\xi+\nabla_\xi V_\zeta.
\]
An analytic construction must specify weighted spaces at
the fan ends and include the variations of the shifts
\(c_a\). Write \(D_u^{\mathrm{ext}}\) for such an extended
Fredholm realization, when it exists. At a regular solution,
the expected dimension after dividing by translations is
\begin{equation}
\label{eq:fs-instanton-index}
\operatorname{vdim}
\mathcal M_{\zeta,B}(\boldsymbol\phi;\mathfrak c)
=\operatorname{ind}(D_u^{\mathrm{ext}})-2.
\end{equation}
Regularization must respect the end data and be defined
either equivariantly or after fixing a translation slice.
An arbitrary fixed domain-dependent perturbation cannot be
followed by an unchanged translation quotient.

When the reduced zero-dimensional spaces are compact and
coherently oriented, denote their signed counts by
\[
n_{B,\zeta}(\boldsymbol\phi;\mathfrak c)
=\#\mathcal M_{\zeta,B}^{0}
(\boldsymbol\phi;\mathfrak c).
\]
Choose homogeneous soliton generators and a generator
\(o_B\) of \(\omega_B\). With the required identifications
of determinant lines and soliton orientation lines, the
candidate contribution is
\begin{equation}
\label{eq:fs-instanton-count}
\Gamma_B=
\sum_{\boldsymbol\phi,\mathfrak c}
n_{B,\zeta}(\boldsymbol\phi;\mathfrak c)
\bigl(e_{\phi_0}\otimes\cdots\otimes e_{\phi_{k-1}}\bigr)
\otimes o_B,
\end{equation}
where the sum uses the chosen finite or completed coefficient
framework, with class weights when needed. The orientation
identifications must make this expression independent of
the auxiliary orientation generators.

To match the preceding algebra, the index and grading
comparison must in particular show that rigid contributions
have \(\sum_a|e_{\phi_a}|=-d_B\), so that
\[
\Gamma_B\in
\bigl(E_B^{\mathcal N}\bigr)^1
=N_B^{-d_B}\otimes\omega_B.
\]
Here \(d_B=|B|-3\) uses the entire marking, including
interior marks. This compatibility is an additional analytic
and grading requirement, not a consequence of specifying
the asymptotic fan. We do not identify an unshifted Maslov
index with the degree in \(E_B^{\mathcal N}\) without this
comparison.

The expected codimension-one degenerations consist of
breaking at an external soliton end, and separation into
several instantons connected by matching internal solitons.
The former should give the coefficient differential. For
the latter to give the previously constructed brackets, the
limiting webs must correspond to the appropriate regular
subdivisions, with internal states contracted by
\(\beta_\sigma\) and the prescribed incidence signs.
Compactness, gluing, the absence of additional strata, and
closure on the selected finite collection of labels remain
analytic requirements. This is a description of the
candidate moduli problem; the analogous conclusions in
\cite[Conjecture~14.10]{KKS} are also formulated
conjecturally.

Suppose these counts are defined and the compactified
one-dimensional moduli spaces have exactly the breaking and
subdivision strata encoded by the coefficient differential
and the \(L_\infty\)-operations. The signed boundary identity
then gives, for
\[
\gamma_{S,\mathfrak d}=\sum_B\Gamma_B
\in\mathfrak g_{\mathcal N}^1,
\]
the Maurer--Cartan equation
\begin{equation}
\label{eq:fs-intrinsic-mc}
\ell_1(\gamma_{S,\mathfrak d})+
\sum_{n\geq2}\frac{1}{n!}
\ell_n(\gamma_{S,\mathfrak d},\ldots,\gamma_{S,\mathfrak d})=0.
\end{equation}
Establishing the counts, their degrees, and this description
of all boundary strata is the analytic part of the proposal.
The sums are taken over the ground field when finite and in
a compatible energy completion when necessary.

\subsection{Comparison with the complex-valued Morse model}
\label{subsec:comparison-complex-valued-morse-model}

The relation to the planar model can be stated without
lifting any path or thimble. The primitives \(W_U\) identify
the local equations with the complex-valued equations;
affine development of the polygonal domain identifies the
subdivision combinatorics. These are different roles of
local coordinates, and neither discards the transport data
on the punctured base.

In the simplest case the polygonal datum and the paths lie
in an embedded coordinate disc \(U\subset E\). The function
\(W_U\) is then defined on \(f^{-1}(U)\). If the relevant
Floer and instanton moduli spaces are confined to this region,
the comparison is exactly the corresponding local planar
problem. For winding paths and overlapping polygon images,
the equations still make sense on \(X\), but confinement,
transport, and the completeness of the chosen finite model
require additional arguments.

The path collection provides affine data through
the periods
\[
q_i=\int_{\delta_i}\eta,\qquad
h_{ij}=[\delta_i^{-1}*\delta_j]\in\Pi_1(E;w_i,w_j),
\qquad
\int_{h_{ij}}\eta=q_j-q_i.
\]
The normalization places the root at (0) in the developing
coordinate. To obtain a relative datum of the preceding
sections, require
\[
0\notin\operatorname{Conv}\{q_1,\ldots,q_N\}
\]
and general position of \(0,q_1,\ldots,q_N\), together with the
factorization property of Definition~\ref{def:factorization-genericity}
for this augmented configuration. The resulting angular order must
agree with the chosen directed order.
These conditions are not automatic for an arbitrary ordered
collection of paths.

Even when they hold, periods specify only the unpunctured
path classes. A direct geometric comparison requires the
actual paths to agree, up to the prescribed admissible
isotopies, with the projected root paths of the polygonal
datum. If this is not the case, the comparison must include
the relevant transport or mutation data. Matching the two
orders alone is insufficient.

There is also a finiteness issue. The finite polygonal
algebra uses only the selected marking. For a geometric
realization, one must check that the relevant moduli spaces
and their compactifications close on these labels. A useful
case to isolate is a domain with
\(\rho^{-1}(A)\cap P^+=M^\circ\), together with analytic
confinement to that domain. More generally, overlapping
images can introduce additional occurrences of critical
values and further path sectors. If these contribute, the
finite algebra must be enlarged; they cannot simply be
discarded. Such an enlargement is additional to the fixed
finite construction of the preceding sections.

With a compatible coefficient system and relative datum,
the formal algebraic connection is explicit. Write
\[
R=R_{\mathcal N,\mathfrak d}^+,
\qquad R_{ii}=\Bbbk e_i,
\]
and let
\[
\Psi:\mathfrak g_{\mathcal N}\longrightarrow
\overrightarrow C^{\geq1}(R,R)[1]
\]
be the directed \(L_\infty\)-action. Its Taylor components
are the full mixed-subdivision maps from the preceding
sections.

\begin{proposition}[Formal deformation associated to instanton counts]
\label{prop:fs-formal-deformation}
If \(\gamma=\gamma_{S,\mathfrak d}\) satisfies
\eqref{eq:fs-intrinsic-mc}, then
\[
b_\gamma=\sum_{n\geq1}\frac{1}{n!}
\Psi_n(\gamma,\ldots,\gamma)
\]
is Maurer--Cartan in the directed Hochschild dg-Lie algebra.
Consequently
\begin{equation}
\label{eq:fs-deformed-operations}
m_k^\gamma=m_k+
\sum_{n\geq1}\frac{1}{n!}
\bigl(\Psi_n(\gamma,\ldots,\gamma)\bigr)_k
\end{equation}
defines a directed, uncurved \(A_\infty\)-algebra
\(R^\gamma\), with the same underlying graded bimodules and
fixed diagonal units. The sums use the stated finite or
complete filtered model.
\end{proposition}

\begin{proof}
An \(L_\infty\)-morphism sends Maurer--Cartan elements to
Maurer--Cartan elements. Hence
\[
[m_1+m_2,b_\gamma]+\tfrac12[b_\gamma,b_\gamma]=0.
\]
Since the original bar coderivation \(m=m_1+m_2\) squares to
zero, the deformed coderivation \(m+b_\gamma\) also squares
to zero. Its components are \eqref{eq:fs-deformed-operations}.
The directed and normalized conditions preserve the order,
the endpoint actions, and the adjoined units. No arity-zero
component occurs. For the fixed finite configuration the
subdivision formulas have bounded arities; analytic sums
inside the coefficients are subject to the stated convergence
conditions.
\end{proof}

This proposition uses the existence of \(\Psi\), not its
being a quasi-isomorphism. Thus the deformation construction
is independent of the unresolved unrestricted universality
claim discussed in Section~\ref{uni-thm}.

The rooted summands explain why the entire algebra \(R\)
is needed in the FS comparison. A rooted polygon with
boundary chain \(i=i_0,\ldots,i_s=j\) away from the root
contributes
\[
F_B^{\mathcal N}=
\bigl(N_{i_0i_1}\otimes\cdots\otimes N_{i_{s-1}i_s}\bigr)
\otimes\omega_B[-d_B-1]
\subset R_{ij}.
\]
Thus \(R_{ij}\) assembles chains of soliton states, with the
polygonal shifts. In particular, even a three-marked-point
rooted triangle contributes its edge coefficient with the
shift prescribed earlier. One should not identify an
unshifted edge complex \(N_{ij}\) with a thimble morphism
complex by itself.

In the proposed degeneration of Floer polygons, these chains
are the intermediate configurations, unrooted instanton
vertices carry the coefficients of \(\gamma\), and mixed
subdivisions assemble them through \(\Psi_n\).
Formula~\eqref{eq:fs-deformed-operations} is precisely the
algebraic expression of that proposed gluing description.
Proving the degeneration and gluing gives the geometric
comparison; the formal deformation identity alone does not.

\subsection{Expected results}

\begin{conjecture}
\label{conj:fs-total-algebra-from-aoi}
Let \(f:X\to E\), \(S\), and \(\mathfrak d\) satisfy the
Floer admissibility, affine compatibility, and finite-model
requirements described above. Use compatible soliton
coefficient complexes, brane gradings, and orientation data.
Then the polygonal instanton model supplies a
Maurer--Cartan element
\[
\gamma_{S,\mathfrak d}\in
\mathfrak g_{\mathcal N}^1
\]
with the grading conventions of the preceding sections.
There is a strictly unital \(A_\infty\)-quasi-isomorphism,
preserving the object idempotents,
\[
\mathcal F_S:
R_{\mathcal N,\mathfrak d}^{+,\gamma_{S,\mathfrak d}}
\longrightarrow R_S.
\]
If convergence requires completion, both sides and the map
are taken in the same filtered coefficient framework.
\end{conjecture}

The intended comparison sends the object labelled by \(m_i\)
to \(L_{\delta_i}\). Its linear component must identify the
deformed soliton-chain complex with the directed Floer
complex on cohomology; its higher components must intertwine
the operations. These are substantive geometric assertions,
not consequences of the universality statement.
The expected equivalence is therefore
\[
\operatorname{Tw}
\bigl(R_{\mathcal N,\mathfrak d}^{+,\gamma_{S,\mathfrak d}}\bigr)
\simeq \operatorname{Tw}(A_S),
\]
where the algebra on the left is viewed as a category using
its object idempotents. It identifies the full category
\(\operatorname{FS}_z(f)\) only when the directed realization
is full and the selected thimbles generate it. The analogous
statement for split closures requires split-generation.

Finally, the dependence on paths can be expressed entirely
on \(E\). Choose loops
\(\gamma_i\in\pi_1(E\setminus A,z)\) for which the modified
paths remain in the chosen admissible class, and write
\[
\delta_i'=\gamma_i*\delta_i,
\qquad
\lambda_i=\int_{\gamma_i}\eta.
\]
The corresponding affine periods change by
\[
q_i'=q_i+\lambda_i,
\qquad \lambda_i\in\Lambda_\eta:=
\left\{\int_\gamma\eta:\gamma\in H_1(E;\mathbb Z)\right\}.
\]
This formula describes the change of the affine input, but
does not specify the change of the geometric coefficients.
The latter also involves the symplectic monodromy along
\(\gamma_i\). A small loop around a critical value has
\(\lambda_i=0\), while its Picard--Lefschetz monodromy may
act nontrivially on vanishing cycles. Thus two path systems
can have identical affine periods and different transport
data relevant to the directed category.

Applying a common exact graded symplectomorphism to all
objects, with transported Floer data, gives an equivalent
directed category. A common base monodromy gives this
conclusion in a fiber model, and in a thimble model when a
compatible lift to that model is supplied. Changing one path
can instead change relative Floer complexes and higher
products. By comparison, translating a developing coordinate
changes no geometric path at all and is not a monodromy
operation.

Accordingly, the expected FS model depends on the intrinsic
polygonal datum together with punctured-path transport and
instanton data. An arbitrary winding collection need not be
represented by one of the finite affine configurations used
here. Further transport, additional sectors, or a completed
family of such models may be needed to recover the whole
curve-valued category.

\printbibliography

\end{document}